\documentclass[11pt,letterpaper,twoside,english]{amsart}
\usepackage{hyperref}
\usepackage[normalem]{ulem}
\usepackage{cleveref}
\usepackage[foot]{amsaddr}
\usepackage{amsmath,amsthm,amssymb}
\usepackage{tikz,tikz-cd}
\usepackage{setspace}
\usepackage{comment}
\usepackage{graphicx}

\usepackage[backend=biber,style=alphabetic,maxnames=5]{biblatex}

\newtheorem{theorem}{Theorem}[section]
\newtheorem{prop}[theorem]{Proposition}
\newtheorem{conjecture}[theorem]{Conjecture}
\newtheorem{lemma}[theorem]{Lemma}
\newtheorem{corollary}[theorem]{Corollary}
\newtheorem{claim}[theorem]{Claim}
\theoremstyle{remark}
\newtheorem{definition}[theorem]{Definition}
\newtheorem{remark}[theorem]{Remark}
\newtheorem{example}[theorem]{Example}

\numberwithin{equation}{section}
\def\bs{\backslash}

\def\bfG{\mathbf{G}}

\def\bC{\mathbb{C}}
\def\bQ{\mathbb{Q}}
\def\bR{\mathbb{R}}
\def\bZ{\mathbb{Z}}
\def\bP{\mathbb{P}}
\def\bV{\mathbb{V}}
\def\cF{\mathcal{F}}
\def\cC{\mathcal{C}}
\def\cM{\mathcal{M}}
\def\cE{\mathcal{E}}

\def\cU{\mathcal{U}}
\def\cO{\mathcal{O}}
\def\cI{\mathcal{I}}
\def\cZ{\mathcal{Z}}
\def\cT{\mathcal{T}}
\def\cP{\mathcal{P}}
\def\cG{\mathcal{G}}
\def\cS{\mathcal{S}}
\def\fC{\mathfrak{C}}
\def\fh{\mathfrak{h}}
\def\tKNU{\mathrm{KNU}}
\usepackage[mathscr]{eucal}

\def\olB{\overline B}
\def\olS{\overline S}
\def\ol{\overline}

\def\olphi{\overline{\Phi}}
\def\End{\mathrm{End}}

\def\tA{\tilde A}

\def\an{\mathrm{an}}
\def\alg{\mathrm{alg}}

\def\psigma{\pmb{\sigma}}
\def\ptau{\pmb{\tau}}

\usepackage{enumitem}

\begin{document}

\title{On the generalized toroidal completion of period mappings}
\author{Haohua Deng}
\email{haohua.deng@dartmouth.edu}
\author{Jacob Tsimerman}
\email{jacobt@math.toronto.edu}

%\date{\today}

\maketitle

\small
\begin{center}
\textbf{Abstract}
\end{center}

Given a period map defined over a quasi-projective variety, we construct a completion with rich geometric and Hodge-theoretic meaning. This result may be regarded as an analog of Mumford's toroidal compactification for locally symmetric Hodge varieties as well as a realizable alternative of Kato--Nakayama--Usui's construction. As an application, we give an alternative proof on the algebraicity of weakly special subvarieties.
%\tableofcontents
% ---------------------------------------
\section{Introduction}\label{Sec01}
% ---------------------------------------

\subsection{Motivation}

Hodge-theoretic methods are fundamental in moduli theory of algebraic varieties. Let $\mathcal{X}\rightarrow B$ be a family of smooth projective varieties over some quasi-projective base manifold $B$. It gives a canonical integral variation of Hodge structure over $B$ which induces a period map $\Phi$ from $B$ to some analytic Hodge variety $\Gamma\backslash D$. One may obtain ample information about the family $\mathcal{X}\rightarrow B$ by studying its associated period map $\Phi$. Therefore, the study of local and global behavior of $\Phi$ including its asymptotic behavior near the boundary of $B$ becomes one of the central topics in Hodge theory and algebraic geometry.

%For any subset $I\subset \mathcal{I}$ we write $Z_I:=\cap_{i\in I}Z_i$. We write $Z_I^{*}:=\cap_{j\notin I}(Z_I-Z_j)$. Moreover we may assume all $Z_I$ are irreducible, as we may always blow up extra components.

Given an abstract period map $\Phi$, a long-investigated problem is to find a good completion $\overline{\Phi}$ of $\Phi$ to a good algebraic compactification $\olB$. This was firstly introduced by Griffiths in \cite[Sec. 9]{Gri70}. Roughly speaking, ``Good compactification" usually means $\olB$ is a log compactification of $B$ (i.e., $\olB-B$ is a simple normal crossing divisor). For the completion $\overline{\Phi}$ to be "good", we usually require the following properties: 
\begin{enumerate}
    \item $\overline{\Phi}$ should be a morphism in some subcategory of the category of complex analytic spaces.
    \item The Hodge-theoretic interpretation of $\overline{\Phi}$ along $\olB-B$ should be clear in the sense that it records the limiting Hodge-theoretic data.
    \item When $\Phi$ is the period map underlying a family of smooth projective varieties $\mathcal{X}\xrightarrow{\pi} B$, we may study singular objects along $\olB-B$ by looking at the boundary limiting Hodge data from $\overline{\Phi}$.
\end{enumerate}

Griffiths' problem is almost perfectly addressed when the Hodge variety $\Gamma\backslash D$ is locally Hermitian symmetric. Related works include \cite{BB66} and \cite{AMRT10}. Though these works on the compactifications of locally symmetric varieties are developed without Hodge-theoretic terms, they can be directly applied to many problems in Hodge theory.

{
For the most general cases, an analogue of the Baily-Borel compactification was recently announced in this generality in \cite{BFMT25}, following earlier works in \cite{BBT22},\cite{GGLR20},\cite{GGR21}. In this paper, we answer this question in a different direction by constructing analogues of toroidal compactifications following \cite{KU08}, \cite{KNU10}.
}

\subsection{The main result}

The purpose of this paper is to provide a solution to Griffiths' long-standing problem. The following informally-stated theorem is the main result of this paper and will be formalized as Theorem \ref{Thm:MainTHMformal}.

\begin{theorem}\label{Thm:MainTHMinformal}
    Denote $\mathrm{Im}(\Phi)=:\wp\subset \Gamma\backslash D$ and let $\Phi': B\rightarrow \wp'$ be the Stein factorization of $\Phi: B\rightarrow \wp$. There exists a projective completion $\olB$ of $B$, a completion $\overline{\wp}$ of $\wp'$ as complex analytic spaces, such that $\Phi$ may be completed as a morphism of compact complex analytic spaces $\overline{\Phi}: \olB\rightarrow \overline{\wp}$. Moreover, up to finite data, $\overline{\Phi}(b)$ records the limiting mixed Hodge structures at $b\in \olB$.
\end{theorem}

The main idea of the proof of Theorem \ref{Thm:MainTHMinformal} is as follows: Fix a completion $\hat{B}$ of $B$, we may find a proper transform $\olB$ of $\hat{B}$ over $B$, a finite open cover $\olB=\bigcup_kW_k$, such that on each $W_k$ we may extend $\Phi|_{B\cap W_k}$ to some $\Phi_k$ defined over $W_k$. This extension is based on Kato--Nakayama--Usui's theory \cite{KU08}, \cite{KNU10}. Moreover, we will show these local extensions can be patched together to a topological map over the whole $\olB$ which is actually a morphism in the category of complex analytic spaces. We will also show that in some special cases, our extended period map is the analytification of a projective morphism. 

As a consequence, the main theorem of this paper may be regarded as a generalization of \cite[Chap. 1-3]{AMRT10} to any period mappings defined over some quasi-projective variety. It is natural to ask whether the projectivity result in \cite[Chap. 4]{AMRT10} also generalizes (Conjecture \ref{Conj:ProjectiveCompletion}).

As an application of this theorem, we give an alternative proof on the algebraicity of weakly special subvarieties. This is a well-known result and was proved by o-minimal methods. Our alternative proof does not use o-minimal theory but the logarithmic Hodge theory and the main theorem.

\subsection{Structure of the paper}

In Section \ref{Sec02} we will go over all of the main ingredients in this paper. Necessary backgrounds in Hodge theory will be reviewed in Section \ref{Sec03}. In Section \ref{Sec04} we will introduce the concept of combinatorial monodromy complex (CMCX) which is crucial for our construction. Section \ref{Sec05} will be devoted to the construction of the extended period map as an analytic morphism. In Section \ref{Sec06} we will discuss applications including some examples in which our construction results in completions in the projective category. In Section \ref{Sec07} we will provide an alternative proof on the algebraicity result for weakly special subvarieties. Finally in Section \ref{Sec08} and \ref{Sec09}, we will prove the general existence of CMCX.

\medskip

\noindent\textbf{Acknowledgment.} The authors thank a lot of people for related helpful discussions, including but not limit to: Ben Bakker, Harold Blum, Robert Friedman, Patricio Gallardo, Phillip Griffiths, Leo Herr, Matt Kerr, Chikara Nakayama, Colleen Robles, Christian Schnell, Salim Tayou. The first author wants to thank Duke University for great environment provided to postdoctoral scholars. The second author wants to thank the Institute for Advanced Study and the Ashvin B.
Chhabra and Daniela Bonafede-Chhabra Membership in particular for their support.

\section{Backgrounds and the main results}\label{Sec02}
% ---------------------------------------

%-----------------------
\subsection{Overview}
%-----------------------

In this section we review the motivation of this paper as well as some relevant works. Let $\bV\rightarrow B$ be a $\bZ$-local system underlying some $\bZ$-polarized variation of (mixed) Hodge structures over some quasi-projective variety $B$. Denote the associated period map as
\[
\Phi: B\rightarrow \Gamma\backslash D
\]
where $\Gamma$ is the monodromy group and $D$ is the corresponding period domain parametrizing polarized Hodge structures of weight $l$ and Hodge numbers $\{h^{p,q}\}_{p+q=l}$. 
%Up to passing to some finite-index normal subgroup of $\Gamma$, we may assume $\Gamma$ is neat. 
Let $\mathbf{G}:=\mathrm{Aut(D)^{+}}$ which is an algebraic group defined over $\bQ$. Let $\mathfrak{g}=\mathrm{Lie}(\mathbf{G})$.

\subsubsection{The classical cases}

We first consider the case when the period domain $D$ is Hermitian symmetric and $\Gamma$ is an arithmetic subgroup of $\mathrm{Aut}(D)$. In geometry, two most common examples are:
\[
l=1, h^{1,0}=h^{0,1}=g, h^{p,q}=0 \ \text{otherwise}, 
\]
or
\[
l=2, h^{2,0}=h^{0,2}=1, h^{1,1}=a\geq0, h^{p,q}=0 \ \text{otherwise}.
\]
In the first case, the period domain $D$ is exactly the generalized Siegel half spaces, or the set of symmetric $2g\times 2g$ matrices with complex coefficients and positive-definite imaginary part.

The locally symmetric variety $\Gamma\backslash D$ has many well-known compactifications, for example, the Satake-Baily-Borel (SBB) compactification \cite{BB66} and the toroidal compactification \cite{AMRT10} which can be described via Hodge-theoretic terms. 

\begin{theorem}[\cite{BB66, AMRT10}]
    Given a locally symmetric variety $\Gamma\backslash D$, there are canonical Satake-Baily-Borel compactification $\overline{\Gamma\backslash D}^{\mathrm{SBB}}$ and toroidal compactifications $\overline{\Gamma\backslash D}^{\Sigma}$ depending on the choice of a toroidal fan $\Sigma$, both in the category of projective varieties. Moreover, there is a reduction morphism $\overline{\Gamma\backslash D}^{\Sigma}\rightarrow \overline{\Gamma\backslash D}^{\mathrm{SBB}}$ which may be regarded as partial desingularization.
\end{theorem}

The SBB and toroidal compactifications are important in the study of moduli spaces of certain geometric objects, especially their degenerations. One of the typical questions in the study of moduli theory is to ask whether a given period mapping $\mathcal{M}\rightarrow \Gamma\backslash D$ over some family of smooth projective varieties over $\cM$ can be extended to some compactification of $\cM$ with the target space to be the SBB or toroidal compactification of $\Gamma\backslash D$.

In the curve case where $l=1$, let $\Gamma=\mathrm{Sp}_{2g}(\bZ)$, the period map $\Phi_g: \cM_g\rightarrow \mathfrak{A}_g:=\Gamma\backslash D$ associated to the universal genus $g$ curve $\cM_{g,1}\rightarrow \cM_g$ is the Torelli map sending a smooth genus $g$ algebraic curve to its Jacobian. The image of $\Phi_g$ is known as the Jacobian locus inside $\mathfrak{A}_g$, the moduli space of rank $g$ principal polarized abelian varieties.

\begin{theorem}[\cite{Nam76, AB12}]
    The Torelli map $\Phi_g: \cM_g\rightarrow \mathfrak{A}_g$ can be extended to regular morphisms from $\overline{\cM_g}$ to the first and second Voronoi compactifications $\mathfrak{A}_g^{\mathrm{Vor}}$ and $\mathfrak{A}_g^{\mathrm{Perf}}$, both of which are toroidal compactifications and descend to a morphism from $\overline{\cM_g}$ to the SBB compactification $\mathfrak{A}_g^{\mathrm{SBB}}$.
\end{theorem}

Besides the curves case, there are other examples showing the power of the SBB compactification and toroidal compactification in the study of moduli theory. For example, see \cite{ABE22} for elliptic K3 surfaces, \cite{ACT11} for cubic threefolds, and \cite{Laz10} for cubic fourfolds.

\subsubsection{Period mappings of general types}

Now let us assume the period map $\Phi$ is not of the classical type. This means the period domain $D$ is no longer Hermitian symmetric. In this scenario, essential obstructions for a parallel transform of the statements in the classical case to non-classical ones arise. 

\begin{theorem}\cite{GRT14}
    If $D$ is non-classical\footnote{The definition of ``classical cases" in \cite{GRT14} is a bit different from the definition we use here.} and $\Gamma\leq \mathrm{Aut}(D)^+$ is an arithmetic subgroup, then the Hodge variety $\Gamma\backslash D$ has no algebraic structure compatible with the analytic structure on $D$.
\end{theorem}

That is to say, it is hopeless to expect the period map $\Phi:B\rightarrow \Gamma\backslash D$ to be completed in the category of algebraic varieties as in the classical cases. To tackle this difficulty, Kato--Nakayama--Usui (KNU) introduce the theory of space of nilpotent orbits \cite{KU08}, \cite{KNU10}. Their theory is a combination of classical asymptotic Hodge theory (e.g. \cite{Sch73}, \cite{CKS86}) and the theory of logarithmic geometry. In informal language, one of their main results is:
\begin{theorem}\label{Thm:KNUmaintheoreminformal}
    Assuming a $\Gamma$-strongly compatible weak fan $\Sigma$ exists, the period map $\Phi:B\rightarrow \Gamma\backslash D$ may be completed to a morphism $\Phi:\olB\rightarrow \Gamma\backslash D_{\Sigma}$ in the category of logarithmic manifolds.
\end{theorem}
See Section \ref{Sec:KNUtheory} for the formal definitions and statements. Regarding the existence of $\Sigma$, well-known results are the following:
\begin{itemize}
    \item When $B$ is a quasi-projective curve, such $\Sigma$ exists and the construction is trivial.
    \item When $D$ is Hermitian symmetric, we may choose $\Sigma$ in a way such that $\Gamma\backslash D_{\Sigma}$ is exactly the toroidal compactification constructed by \cite{AMRT10}. See \cite{Hay14} for a comparison between these two compactifications in the Siegel space case.
    \item When the monodromy group $\Gamma$ is abelian, such $\Sigma$ exists \cite[Chap. 4]{KU08}.
\end{itemize}

Recently, Colleen Robles and the author showed that such object $\Sigma$ also exists when $B$ is a quasi-projective surface. The construction is in general highly non-trivial:
\begin{theorem}\cite{DR23}\label{Thm:DengRoblesMainTheorem}
    When $B$ is a quasi-projective surface, there exists a projective completion $\olB(\Sigma)$ and a KNU space $\Gamma\backslash D_{\Sigma}$ such that the period map $\Phi: B\rightarrow \Gamma\backslash D$ can be completed to $\overline{\Phi}_{\Sigma}: \olB(\Sigma)\rightarrow \Gamma\backslash D_{\Sigma}$ in the category of logarithmic manifolds, and $\mathrm{Img}(\overline{\Phi}_{\Sigma})$ admits a structure of compact complex algebraic space.
\end{theorem}
However, in \cite{DR23} the authors were not able to show the similar results when $B$ has dimension $\geq 3$ due to some essential technical obstructions. 

%-----------------------
\subsection{Main results}
%-----------------------

The main purpose of this paper is to construct a completion for a general period map with significant geometric and Hodge-theoretic meaning. By geometric we mean this completion should be done in some subcategory of complex analytic spaces, while by Hodge-theoretic we mean the extended period map should record the boundary asymptotic behaviors of the original period map.

More precisely, given a period map $\Phi: B\rightarrow \Gamma\backslash D$, one may factor the period map $\Phi$ via the diagram:
\begin{equation}
    B\xrightarrow{\Phi'} \wp'\rightarrow \wp:=\mathrm{Img}(\Phi)\hookrightarrow\Gamma\backslash D
\end{equation}
where $ B\xrightarrow{\Phi'} \wp'$ is the Stein factorization of $\Phi$ and is the analytification of a morphism in the category of quasi-projective varieties \cite{BBT22}. Instead of finding a completion of $\Phi:B\rightarrow \Gamma\backslash D$ which requires an enlarged space of the entire Hodge variety $\Gamma\backslash D$, we will complete $\Phi': B\rightarrow \wp'$ instead.

One of the main ingradients in our construction is the combinatorial monodromy complex (CMCX), see Section \ref{Sec04}. Roughly speaking, fix a completion $\hat{B}$ of $B$, a CMCX $\cT$ is a collection of rational polyhedral complexes indexed by the strata of $\hat{B}$. For every boundary stratum $Z_I^*$ of $\hat{B}$, there is a canonically associated combinatorial monodromy cone (CMC) $\psigma_I$ (Section \ref{Sec:CMC}), and $\cT$ associates each $I$ a polyhedral complex $\Sigma_{\cT}(I)$ supported on $\overline{\psigma_I}$ with some compatibility conditions (Definition \ref{def:CMCX}).

\begin{remark}
CMCX is an analog of a toroidal fan used by \cite{AMRT10} or a $\Gamma$-strongly compatible weak fan used by \cite{KU08} and \cite{KNU10} which does not require any compatibility conditions with the global monodromy group $\Gamma$.
\end{remark}

The construction of a CMCX depends on both the $\bZ$-polarized variation of Hodge structures and a choice of projective completion $\hat{B}$ of $B$ with simple normal crossing (SNC) divisors. Assuming both ingredients, we have:
\begin{theorem}\label{Thm:CMCXexists}
For a fixed $\bZ$-PVHS $\bV\rightarrow B$ over a quasi-projective $B$ and a SNC completion $\hat{B}$ of $B$, a combinatorial monodromy complex always exists. 
\end{theorem}
Theorem \ref{Thm:CMCXexists} will be proved in general in Section \ref{Sec09}. The techniques are purely combinatorial. In Section \ref{Sec04} we will see that under some mild assumptions, the construction of the CMCX is evident.

Now we may suppose a CMCX $\cT$ corresponds to the $\bZ$-PVHS $\bV\rightarrow B$ and a completion $\hat{B}$ of $B$ is given. Let $\olB\rightarrow \hat{B}$ be a logarithmic modification of the pair $(\hat{B}, B)$ (See Section \ref{Sec04} for details). We may also define the compatibility between $\olB\rightarrow \hat{B}$ and $\cT$ (Definition \ref{def:CMCXcompatible}). For a given $\cT$ a compatible logarithmic modification always exists (Theorem \ref{Thm:CMCXgiveslogmodification}). We are now ready to introduce the main theorem of this paper.

\begin{theorem}\label{Thm:MainTHMformal}
    Let $B\xrightarrow{\Phi'}\wp'\rightarrow \wp$ be the Stein factorization of $B\xrightarrow{\Phi}\wp$ and $\hat{B}$ be a fixed completion of $B$ with SNC boundary divisors. Let $\cT$ be a combinatorial monodromy complex associated to $\hat{B}$. There exists a completion $\olB$ of $B$ compatible with $\cT$, a completion $\overline{\wp}$ of $\wp'$, such that we have the following extension in the category of complex analytic spaces (and consequently, algebraic spaces):
    \begin{equation}\label{Fig:extendedperiodmap}
\begin{tikzcd}
B\arrow[d] \arrow[r, "\Phi'"] & \wp' \arrow[d]\\
\olB \arrow[r, "\overline{\Phi}"] & \overline{\wp}
\end{tikzcd}
\end{equation}
Moreover, we have the following properties:
\begin{enumerate}
    \item We may choose $\olB$ to be smooth projective with $\olB-B$ being a simple normal crossing divisor;
    \item For any $b'\in \olB$ lying above $b\in Z_I^*\subset \hat{B}$, there exists a combinatorial monodromy cone $\ptau_{b'}\in \Sigma_{\cT}(I)$, a nilpotent orbit $(\tau_{b'}, F_{b'})$ with $\tau_{b'}\simeq \ptau_{b'}$, such that $\overline{\Phi}(b')$ records the information of $(\tau_{b'}, F_{b'})$ up to some finite data.
\end{enumerate}
\end{theorem}
\begin{remark}
    Indeed, the nilpotent orbit $(\tau_{b'}, F_{b'})$ in Theorem \ref{Thm:MainTHMformal} is determined by the classical theory of the nilpotent orbit theorem by Schmid in \cite{Sch73}. See Section \ref{Sec03} for an overview.
\end{remark}
{
\begin{remark}
     In particular, this gives a new proof of the algebraicity of period images, proven in \cite{BBT22}, up to normalization.
\end{remark}
}

The proof of Theorem \ref{Thm:MainTHMformal} will be done in two steps: In Section \ref{Sec05} we construct the map $\overline{\Phi}: \olB\rightarrow \overline{\wp}$ on the set-theoretic level and show $\overline{\Phi}: \olB\rightarrow \overline{\wp}$ has a natural structure in the category of complex analytic spaces.
In the most general cases, describing the construction of $\overline{\Phi}$ and the completion $\olB$ of $B$ could be complicated on the computational side. However, under some certain assumptions the constructions are elementary:

\begin{theorem}[Theorem \ref{Thm:simplicialCMCX}]
    If the completion $\olB$ of $B$ satisfies the simplicial condition with respect to the $\bZ$-PVHS $\bV\rightarrow B$, then we may take $\hat{B}=\olB$ and $\overline{\Phi}$ maps any $b\in \olB$ to its associated nilpotent orbit class given by Schmid's nilpotent orbit theorem. In particular, this always holds when $B$ is a quasi-projective curve or surface.
\end{theorem}

Motivated from the construction of \cite{AMRT10} and the quasi-projectivity of period images \cite{BBT22}, it is natural to conjecture:

\begin{conjecture}\label{Conj:ProjectiveCompletion}
    The construction in Theorem \ref{Thm:MainTHMformal} can always be done in a way such that $\overline{\wp}$ is a projective variety.
\end{conjecture}

In Section \ref{Sec07} we will see some examples in which this conjecture holds.

\subsection{Applications and related works}

We introduce some applications of our main result as well as its interaction with several related works.

\subsubsection{Comparison with \cite{DR23}}

Suppose now $B$ is a quasi-projective surface. The completion
\[\overline{\Phi}_{\Sigma}: \olB(\Sigma)\rightarrow \Gamma\backslash D_{\Sigma}
\]
given by Theorem \ref{Thm:DengRoblesMainTheorem} in \cite{DR23} is defined on some logarithmic modification $\olB(\Sigma)$ of $\olB$. Due to the fact that the global monodromy group $\Gamma$ is mostly immeasurable, it is in general not possible to describe the weak fan $\Sigma$ or the projective surface $\olB(\Sigma)$ over which the map $\overline{\Phi}_{\Sigma}$ lives in.

The Theorem \ref{Thm:MainTHMformal} and Theorem \ref{Thm:simplicialCMCX} gives another completion
\[
\overline{\Phi}:\olB\rightarrow \overline{\wp}
\]
which does not require the existence of any extension of the entire Hodge variety $\Gamma\backslash D$. In this case, the completion $\olB$ can be any smooth completion of $B$ with $\olB-B$ a normal crossing divisor. 

% \begin{theorem}
%     The following diagram exists and belongs to the category of complex analytic spaces\footnote{Strictly speaking, we need to proceed the Stein factorization on $\overline{\Phi}_{\Sigma}$.}:
%     \begin{equation}
%     \begin{tikzcd}
%         \olB(\Sigma) \arrow[d] \arrow[r, "\overline{\Phi}_{\Sigma}"] \arrow[dr, dashed] & \overline{\wp}(\Sigma)\subset \Gamma\backslash D_{\Sigma}\arrow[d]\\
%         \olB \arrow[r, "\overline{\Phi}"] & \overline{\wp}
%     \end{tikzcd}
%     \end{equation}
% \end{theorem}
As a summary, though the main theorem of this paper does not provide a completed Hodge variety $\Gamma\backslash D_{\Sigma}$ which plays an important role in KNU's theory \cite{KU08}, \cite{KNU10}, in many cases the completion $\olB$ over which the extended (Stein-factorized) period map exists is much easier to describe.

% \subsubsection{Comparison with \cite{DR25}}
% {\color{red} Let me know how you want to deal with tihs section, which should be walked back to some extent probably?}

% In \cite{DR25}, Colleen Robles and the first author will give a similar construction. It is known by the Schmid's nilpotent orbit theorem (Theorem \ref{Thm:SchmidNilp}) that for any $b\in \olB$, there exists a unique $\Gamma$-class of nilpotent orbits $(\sigma_b, F_b)$ associated to $b\in \olB$ via the asymptotic behavior of $\Phi$ near $b$. Instead of trying to record the information on both the nilpotent cone $\sigma_b$ and the limiting Hodge flag $F_b$, we only try to have $\overline{\Phi}$ encode the nilpotent orbit $\exp(\sigma_{b, \bC})F_b$. In other words, the information on the submonoid $\sigma_b\subset \sigma_{b, \bC}$ is lost.

% The advantage of the construction in \cite{DR25} is that we do not need a combinatorial monodromy complex which encodes the information about monodromy cones. In particular, this construction can be done over any logarithmic completion $\olB$ of $B$. However, in this case the logarithmic structure of the completed period map would be lost in general.

\subsubsection{Relation with \cite{BFMT25}}

In this paper the authors construct a canonical completion $\overline{\Phi}: \olB\rightarrow \overline{\wp}$ of $\Phi: B\rightarrow \wp$. The construction is called as the generalized (Satake)--Baily--Borel compactification, as the ``hodge-theoretic data" that is recorded at the boundary is only that of pure hodge structures (the graded components) with no extension data. It recovers the classical Satake--Baily--Borel compactification in the case when the period domain $D$ is Hermitian symmetric. As an analog of the relation between Satake--Baily--Borel compactification and toroidal compactification in the classical case, studying the relationship between the generalized completions is a potentially interesting topic.

In \cite{BFMT25} the authors compactify the actual period image, without needing to pass to the Stein factorization. It seems very plausible the same thing is possible with out main theorem.

\subsubsection{Examples for projective completions}

Though the construction of the CMCX $\cT$ and the completion $\hat{B}$ of $B$ given by Theorem \ref{Thm:MainTHMformal} are in general unapproachable without using computer, there are still some accessible cases including the case when $B$ is a surface. In Section \ref{Sec07} we will review several such examples. We will show the main Theorem \ref{Thm:MainTHMformal} gives completed period mappings in the category of projective varieties. These are cases in which the Conjecture \ref{Conj:ProjectiveCompletion} is true.

\subsubsection{Algebraicity of weakly special subvarieties}

%Fix a weak Mumford-Tate domain $M$ with $\mathbf{H}:=\mathrm{Aut}(M)^+$. Let $\Gamma_{H}:=\Gamma\cap H_{\bQ}$. A weakly special subvariety associated with the weak Hodge datum $(\mathbf{H}, M)$ is an irreducible component of $\Phi^{-1}(\Gamma_H\backslash M)$ in $B$. It is an analytic subvariety of $B$.

In Hodge theory, weakly special subvarieties for a $\bZ$-PVHS $\bV\rightarrow B$ are analytic subvarieties of $B$ carrying special Hodge-theoretic or monodromic data. In the lanaguage of o-minimal geometry, since the period map is definable by \cite{BKT20}, weakly special subvarieties are definable analytic subvarieties of $B$ and hence must be algebraic by the definable GAGA theorem \cite{BBT22}. This recovers and generalizes the classical result of \cite{CDK95} that Hodge loci is a countable union of algebraic varieties. 

As an application of our main theorem, we give an alternative proof of the algebraicity result on weakly special subvarieties using our main theorem \ref{Thm:MainTHMformal}. Our proof is based on the Hodge-theoretic interpretation of the completed period map $\olphi$ provided by Theorem \ref{Thm:MainTHMformal}. We will show that weakly special subvarieties of $B$ can be extended to analytic subvarieties of $\olB$ over which an extended period map $\olphi$ exists. The boundary Hodge-theoretic data are described by the period torsors defined by Kato-Usui in \cite[Chap. 3]{KU08}. Indeed, this can be understood as a logarithmic geometry formalism of \cite{CDK95}'s idea.

\subsection{Important remark on the connectivity of strata}

Similar to some related works like \cite{GGLR20} and \cite{GGR21}, we implicitly use the assumption that every strata $Z_I^{*}\subset \hat{B}-B$ is connected. This is in general not the case, but we claim that we may assume it without loss of generality. Indeed, every related argument in this article only uses the partially ordered set of connected boundary strata of $\hat{B}$. In other words, we may re-index each $Z_I^{*}$ by its connected components $Z_{I,1}^{*}, Z_{I,2}^{*},...$ without essentially changing any arguments, but this will make notations in this paper unnecessarily complicated. Therefore, we decide to ignore the component index in this paper.
% ---------------------------------------
\section{Hodge theory background}\label{Sec03}
% ---------------------------------------

In this section some important definitions and results in Hodge theory will be reviewed. We will also introduce the definition of combinatorial monodromy complex which is essential for this paper.

%--------------------------
\subsection{Basic settings}
%--------------------------

Let $B$ be a smooth quasi-projective variety, $\hat{B}$ be projective with $Z:=\hat{B}- B=\bigcup_{i\in \mathcal{I}}Z_i$ being a simple normal crossing divisor with $Z_i$'s to be its smooth irreducible components. Moreover, for any $I\subset \mathcal{I}$, denote $Z_I:=\bigcap_{i\in I}Z_i$ and $Z_I^*:=Z_I-\bigcup_{j\notin I}Z_j$.

Assume there is a $\mathbb{Z}$-local system $\mathbb{V}$ on $B$ which carries an integral polarized variation of Hodge structures ($\bZ$-PVHS). Fix a reference point $b_0\in B$ and denote the fiber of $\bV$ at $b_0$ as $V$, there is a period map associated with the $\bZ$-PVHS:
\begin{equation}\label{periodmapgeneralform}
    \Phi: B\rightarrow \Gamma \backslash D
\end{equation}
 where $D$ is the classifying space of polarized Hodge structures defined on the integral polarized lattice $(V_{\mathbb{Z}}, Q)$ and $\check{D}$ be its compact dual. Moreover, denote $\mathbf{G}=\mathrm{Aut}(H, Q)$ and $\mathfrak{g}:=\mathrm{Lie}(G)$, then $\Gamma\leq G_{\mathbb{Z}}$ is the monodromy group of the local system $\bV\rightarrow B$:
 \[
 \Gamma:=\mathrm{Img}\{\pi_1(B, b_0)\rightarrow \mathrm{Aut}(V_{ \bZ}, Q)\}.
 \]
 In addition, we call $\Gamma\backslash D$ a Hodge variety associated to the Hodge datum $(\mathbf{G}, D)$.

 The differential of the period map $\Phi$ can be interpretated as follows. For any $b\in B$, consider the universal lift $\tilde{\Phi}:\tilde{B}\rightarrow D$ of $\Phi$ and any $\tilde{b}\in \tilde{B}$ lying above $b$, the image $\tilde{\Phi}(\tilde{b})\in D$ may be represented by a Hodge flag $F_b:=\{F_b^k\subset V_{b_0, \bC}\}_{k\in \bZ}$. The infinitesimal variation of Hodge structures (IVHS) $d\Phi_b$ is given by the map:
 \[
 d\Phi_b: T_bB\rightarrow \oplus_{k\in \bZ}\mathrm{Hom}(F_b^k, V_b/F_b^k)
 \]
whose image lies in $\oplus_{k\in \bZ}\mathrm{Hom}(F_b^k, F_b^{k-1}/F_b^k)$ because of the Griffiths transversality.
\subsection{Nilpotent cones and nilpotent orbits}

{Suppose a convex polyhedral cone $\sigma\subset \mathfrak{g}_{\mathbb{Q}}$ is the convex hull of commuting nilpotent elements $\{N_1,...,N_n\}$. In other words,
\begin{equation}
    \sigma:=\{\sum_{j=1}^{n}x_jN_j \ | \ x_1,...,x_n\in \mathbb{Q}_{\geq 0}\}.
\end{equation}}
%\footnote{As defined, in this paper we assume all nilpotent cones are relatively open unless specified.}

\begin{definition}\label{definitionnilporbit}
We say $\sigma$ is a nilpotent cone if there exists some $F\in \check{D}$ such that the following conditions hold:
\begin{enumerate}
\item $NF^{p}\subset F^{p-1}$ for any $N\in \sigma$ and $0\leq p\leq l$;
\item $\mathrm{exp}(\sum_{j=1}^{n}z_jN_j)F\in D$ for $z_j\in \mathbb{C}$ and $\mathfrak{Im}(z_j)\gg0$ for any $j$.
\end{enumerate}
In this case we also call $\cZ:=\exp(\sigma_{\mathbb{C}})F\subset \check{D}$ a nilpotent orbit , denoted as $\cZ:=(\sigma, F)$. Denote $\tilde{A}(\sigma)\subset \check{D}$ as the set of points $F\in \check{D}$ such that $(\sigma, F)$ is a nilpotent orbit, and $A(\sigma):=\exp(\sigma_\bC)\backslash \tilde{A}(\sigma)$. 
\end{definition}
It is a classical result in Hodge theory that the notions of nilpotent orbits and limiting mixed Hodge structures are closely connected: 
\begin{theorem}\cite{CK82}\label{Thm:nilporbittolmhs}
     $(\sigma, F)$ gives a nilpotent orbit if and only if $(W(\sigma), F)$ is a mixed Hodge structure polarized by any $N\in \sigma^{\circ}$, where $W(\sigma)$ is the Jacobson-Morozov weight filtration associated to any $N\in\sigma^{\circ}$, appropriately shifted.
 \end{theorem}

 We thus see that $\tilde{A}(\sigma)$ is naturally a mixed period domain, and the space of nilpotent orbits $A(\sigma)$ is a natural quotient of it. Indeed, $A(\sigma)$ is a \textit{mixed hodge domain} as in \cite[Thm 6.4]{kl17}.

 More precisely, we have the following correspondence of nilpotent orbits and limiting mixed Hodge structures (LMHS):
 \begin{equation}
     \{(\sigma, F)\} \leftrightarrow (W(\sigma), \exp(\sigma_\bC)F\subset \check{D}).
 \end{equation}
Any element $g\in \mathbf{G}$ acts on the set of nilpotent orbits by:\[
 g(\sigma, F)\rightarrow (\mathrm{Ad}_g\sigma, gF)\in A(\mathrm{Ad}_g\sigma).
 \]
\subsubsection{Schmid's nilpotent orbit theorem}

For any $b\in \hat{B}$, choose a local neighborhood around $b\in U\subset \hat{B}$ satisfying $U\cap B\cong (\Delta^*)^k\times \Delta^{n-k}$ on which we may consider the local monodromy operators $\{T_i\}_{1\leq i\leq k}$ around each boundary divisor and their logarithms $\{N_i:=\log(T_i)\}_{1\leq i\leq k}$. By picking a local lift to $D$, we may consider the $N_i$ as elements of $\mathfrak{g}$. One of the main ingredients is the following interpretation of Schmid's nilpotent orbit theorem \cite{Sch73}:
\begin{theorem}\label{Thm:SchmidNilp}
 Pick a local lift to $D$ as above. Then The cone $\mathrm{Conv}(\{N_i\}_{1\leq i\leq k})\subset \mathfrak{g}_\bQ$ is a nilpotent cone, and there is a nilpotent orbit $(\sigma_b, F_b)$ canonically associated to $b$ and our local lift for some $F_b\in \check{D}$. Changing the local lift changes the nilpotent orbit by the action of the corresponding element of $\Gamma$. 
 \end{theorem}
When looking at $\hat{B}$ globally, we have the following corollary:
\begin{corollary}\label{cor:schmidNOT}
    For any $b\in \hat{B}$ there exists a canonically associated nilpotent orbit $(\sigma_b, F_b)$ well-defined up to the action of $\Gamma$.
\end{corollary}
\noindent There are two ways to intepretate nilpotent orbit $(\sigma_b, F_b)$:
\begin{enumerate}
    \item Suppose $\sigma_b=\langle N_1,...,N_r\rangle$, for local coordinates $\{t_i\}$ around $b$, the map $\psi(t_i)=\exp(\sum -\frac{\log(t_i)}{2\pi i}N_i)F_b$ approximates the local lift of $\varphi$ at $b$ (\cite[Sec. 4]{Sch73}).
    \item By Theorem \ref{Thm:nilporbittolmhs}, for every $b\in \hat{B}$ and every $N\in \sigma_b$ there is a monodromy weight filtration $W_b:=W(N)$ associated to $b$, such that $(W_b, F_b, N)$ is an LMHS.
\end{enumerate}
We give a better interpretation of (1) above. Suppose $U\subset \hat{B}$ is an open subset such that $U\cap S\cong (\Delta^{*})^k\times\Delta^l$. Let $\mathfrak{H}$ be the upper-half plane $\{\mathfrak{Im}(z)>0\}$. The local period map $\Phi$ on $U\cap S$ has the local lift:
\begin{equation}
\begin{tikzcd}
\mathfrak{H}^k\times \Delta^l\arrow[d, ""] \arrow[r, "\tilde{\Phi}"] & D \arrow[d] \\
(\Delta^{*})^k\times\Delta^l \arrow[r, "\Phi"] & \Gamma \backslash D
\end{tikzcd}
\end{equation}
If we denote $z=(z_k), w=(w_l)$ as coordinates on $\mathfrak{H}^k$, $\Delta^l$,
\begin{equation}\label{eqn:localliftperiodmap}
    \tilde{\Phi}=\exp(\sum_{1\leq j\leq k}z_jN_j)\psi(e^{2\pi iz}, w),
\end{equation}
where $\psi(z,w)\in D$ is holomorphic over $\Delta^{k+l}$. The second part of Schmid's nilpotent orbit theorem says:
\begin{theorem}\label{Thm:SchmidNilp2}
    For any $w\in\Delta^l$, $(\sigma:=\langle N_1,...,N_k\rangle, \psi(0, w))$ is a nilpotent orbit, and under the canonical metric $d$ on $D$, as $\mathfrak{Im}(z)\rightarrow \infty$,
    \begin{equation}
        d(\exp(\sum_{1\leq j\leq k}z_jN_j)\psi(0, w), \tilde{\Phi}(e^{2\pi iz},w)) \sim e^{-b\mathfrak{Im}(z)}
    \end{equation}
    for some $b>0$.
\end{theorem}

%----------------------------------------------
\subsection{Kato--Nakayama--Usui's construction}\label{Sec:KNUtheory}
%----------------------------------------------

In this section, we review Kato--Nakayama--Usui's theory on the space of nilpotent orbits.

Let $\Sigma$ be a collection of rational nilpotent cones in $\mathfrak{g}_{\mathbb{Q}}$, and $|\Sigma|$ be its support in $\mathfrak{g}_{\mathbb{Q}}$. For $\sigma\in \Sigma$, set $\Gamma(\sigma):=\Gamma\cap \mathrm{exp}({\mathbb{Q}_{\geq 0}}\sigma)$. This is a semigroup. Denote $\Gamma_{\sigma}:=\Gamma(\sigma)^{\mathrm{gp}}$ as its associated group.

\begin{definition}\label{defstronglycompatible}
We say $\Sigma$ is strongly compatible with $\Gamma\leq G_{\mathbb{Z}}$ if the following holds:
\begin{enumerate}
\item $\Sigma$ is closed under the adjoint action by $\Gamma$;
\item $\overline{\sigma_{\mathbb{Q}}}= \mathbb{Q}_{\geq 0}\langle\mathrm{log}(\Gamma(\sigma))\rangle$ for any $\sigma\in \Sigma$.
\end{enumerate}
\end{definition}

\begin{definition}[\cite{KU08}, \cite{KNU10}]\label{deffanweakfan}
We say $\Sigma$ is a fan (resp. weak fan) if the following (1) and (2) (resp. (1) and (2')) hold:
\begin{enumerate}
\item $\Sigma$ is closed under taking faces;
\item For any $\sigma, \tau\in \Sigma$, either $\sigma^{\circ}\cap \tau^{\circ}=\emptyset$ or $\sigma=\tau$;
\item[(2')] For any $\sigma, \tau\in \Sigma$, suppose $\sigma^{\circ}\cap \tau^{\circ}$ is not empty, and there exists $F\in \check{D}$ such that both $(\sigma, F)$ and $(\tau, F)$ are nilpotent orbits, then $\sigma=\tau$.
\end{enumerate}
\end{definition}
Now suppose $\Sigma$ is a $\Gamma$-strongly compatible weak fan. The space of nilpotent orbits is defined as follows:
\begin{align}
    D_{\Sigma}:&=\{(\sigma, Z) \ | \ \sigma\in \Sigma, Z \ \text{is a} \ \sigma-\text{nilpotent orbit}  \}\\
   \nonumber &=\bigsqcup_{\sigma\in \Sigma}A(\sigma)
\end{align}

The first main result of Kato--Nakayama--Usui's theory is about the (partial) completion of the Hodge variety $\Gamma\backslash D$:
\begin{theorem}\cite[Theorem A]{KU08, KNU10}\label{Thm:kumainthmA}
Suppose $\Sigma$ is a fan or weak fan strongly compatible with $\Gamma\leq G_{\mathbb{Z}}$ which is assumed to be neat, then: 
 \begin{enumerate}
 \item $\Gamma \backslash D_{\Sigma}$ admits a structure of logarithmic manifold as well as Hausdorff topological space.
 \item For any $\sigma\in \Sigma$, let  $F(\sigma)$ denote the fan consisting of all the faces of $\sigma$, and $D_\sigma:=D_{F(\sigma)}$. Then the canonical map $\Gamma_{\sigma}\backslash D_{\sigma}\rightarrow \Gamma\backslash D_{\Sigma}$ is a local homeomorphism.
 \end{enumerate}
\end{theorem}
We refer readers to \cite[Chap. 3]{KU08} for the literature of logarithmic manifolds. For our purposes, it suffices to note that these are locally ringed spaces equipped with an fs log-structure. Some important properties will be reviewed in Section \ref{Sec08}. When $D$ is Hermitian, one may choose the (weak) fan $\Sigma$ to be the rational polyhedral decompositions constructed by classical theories (e.g., \cite{AMRT10}, \cite{Nam76}), then the Kato--Nakayama--Usui's construction coincide with the toroidal compactifications $\overline{\Gamma\backslash D}^\Sigma$. In general, the space $\Gamma\backslash D_{\Sigma}$ does not admit a structure of complex varieties. 

The second main result of Kato--Nakayama--Usui's theory is about the moduli properties of the space $\Gamma\backslash D_\Sigma$. 

\begin{definition}\cite[Sec. 2.5]{KU08}\label{def:compatibilityweakfan}
    Given $\Gamma\leq G_\bZ$ and $\Sigma$ which is a (weak) fan strongly compatible with $\Gamma$. We say a $\bZ$-PVHS $\bV\rightarrow B$ and a partial completion $B'$\footnote{$B'$ need not to be in the category of complex spaces. It can be chosen in a more general category $\mathcal{B}(\mathrm{log})$, see \cite[Chap. 2]{KU08}.} is compatible with $(\Gamma, \Sigma)$ if:
    \begin{enumerate}
        \item The monodromy group of $\bV\rightarrow B$ is contained in $\Gamma$;
        \item For any $b\in B'$ and any $\Gamma$-representative of nilpotent orbits $(\sigma_b, F_b)$, there exists a unique minimal $\tau_b\in \Sigma$ such that $\sigma_b\subset \tau_b$ and $(\tau_b, F_b)$ is a nilpotent orbit.
    \end{enumerate}
    In this case, we say the $\bZ$-PVHS $\bV\rightarrow B$ may be extended to a $\bZ$-logarithmic PVHS ($\bZ$-LPVHS) on $B'$ compatible with $(\Gamma, \Sigma)$.
\end{definition}
\begin{theorem}\cite[Theorem B]{KU08, KNU10}\label{Thm:kumainthmB}
The space $\Gamma\backslash D_\Sigma$ represents the moduli functor of $\bZ$-LPVHS compatible with $(\Gamma, \Sigma)$ in the sense that the set of $\bZ$-LPVHS compatible with $(\Gamma, \Sigma)$ has a one-to-one correspondence with the set $\mathrm{Mor}(-, \Gamma\backslash D_\Sigma)$ in the category of logarithmic manifolds.
\end{theorem}
A direct consequence is:
\begin{corollary}\label{cor:extendingperiodmapKNU}
    If the $\bZ$-PVHS $\bV\rightarrow B$ can be extended to a $\bZ$-LPVHS on $B'$ compatible with $(\Gamma, \Sigma)$, then there exists a unique morphism of logarithmic manifolds
    \[
    \overline{\Phi}: B'\rightarrow \Gamma\backslash D_\Sigma
    \]
    extending $\Phi: B\rightarrow \Gamma\backslash D$ to $B'$. For any $b\in B'$, $\overline{\Phi}(b)\in \Gamma\backslash D_{\Sigma}$ is given by the nilpotent orbit $(\tau_b,F_b)$ for some $\tau_b\in \Sigma$ determined by Definition \ref{def:compatibilityweakfan}.
\end{corollary}
The morphism given by the Corollary \ref{cor:extendingperiodmapKNU} only exists in the category of logarithmic manifolds, but under mild assumptions, the image of $\overline{\Phi}$ is a classical object:
\begin{theorem}\label{Thm:analyticimagepropermap}
    Let $U\subset B'$ be an analytic (resp. algebraic) subspace. If $\overline{\Phi}|_U$ is proper, then $U\xrightarrow{\overline{\Phi}}\overline{\Phi}(U)$ is a morphism in the category of complex analytic (resp. algebraic) spaces. 
\end{theorem}
\begin{proof}
    This theorem is in \cite[Sec. 5]{Usu06}. Though the original theorem takes the assumption that $U$ is compact (for example, take $U=B'=\hat{B}$), the proof only uses the properness of $\overline{\Phi}|_U$ and hence applies to our case verbatim. 
\end{proof}

\section{The combinatorial monodromy complex (CMCX)}\label{Sec04}
%----------------------------------------------------

One important ingredient in our proof is the combinatorial monodromy complex. The term ''combinatorial" means we only record the combinatorial structure of the monodromy cone and ignore the structure of monodromy cones inherited from the Lie algebra containing it.

\subsection{The combinatorial monodromy cone (CMC) and adjacency map}\label{Sec:CMC}

For any $I\in \mathcal{I}$, we may choose any point $b\in Z_I^*$, a neighborhood $b\in E_b\subset \hat{B}$ such that $E_b^{*}:=E_b\cap B$ is isomorphic to $(\Delta^*)^{|I|}\times \Delta^{\mathrm{dim}(B)-|I|}$, and a nearby point $b'\in E_b^*$. For every $i\in I$ there is a monodromy logarithm $N_i\in \mathrm{End}(V_{b', \bQ})\simeq \mathfrak{g}_\bQ$ around the divisor $Z_i$. Let $\sigma_{I,b'}:=\mathrm{Conv}\{N_i, i\in I\}^{}$. There exists a linear map:
\[
\lambda_I:\mathrm{span}_\bQ\{\epsilon_i, i\in I\}\simeq\bQ^{|I|}\rightarrow \bQ\sigma_{I,b'} \ , \ \epsilon_i\xrightarrow{\lambda_i}N_i. 
\]
We denote the kernel of $\lambda_I$ as a linear map by $\mathrm{Ker}(\lambda_I)$ which is a subspace of $\bQ^{|I|}$. 
%On the other hand, we may also restrict $\lambda_I$ on $\bQ_{\geq 0}^{|I|}\rightarrow \bQ_{\geq 0}\sigma_{I,b'}$ and regard it as a morphism of monoids. Denote the kernel of this monoid morphism as $\mathrm{Ker}(\lambda_I^M)$ which is a submonoid of $\bQ_{\geq 0}^{|I|}$. 

It is clear that the monodromy cone $\sigma_{I,b'}$ is isomorphic to the image of $\bQ_{\geq 0}^{|I|}$ under the quotient map $\bQ^{|I|}\rightarrow \bQ^{|I|}/\mathrm{Ker}(\lambda_I)$. Its combinatorial structure is determined by the index set $I\subset \mathcal{I}$ and $\mathrm{Ker}(\lambda_I)$. It does not depend on the choice of $b\in Z_I^*$, the neighborhood $E_b$ around $b$, or the nearby point $b'$. 

\begin{definition}\label{def:CMC]}
    The combinatorial monodromy cone (CMC) $\psigma_I$ is defined to be the image of $\bQ_{\geq 0}^{|I|}$ under the map $\bQ^{|I|}\rightarrow \bQ^{|I|}/\mathrm{Ker}(\lambda_I)$.
\end{definition}
\begin{remark}
Starting from this point, for every $I\in \cI$ we will denote $\sigma_I$ or $\tau_I$ as the original monodromy cones as a subset of some $\mathrm{End}(V_{b, \bQ})$, while the corresponding boldface characters $\psigma_I$ or $\ptau_I$ as the corresponding CMC's. Note that as combinatorial cones, $\sigma_I\simeq \psigma_I^{}$.
\end{remark}

For any sub-index $I'\subset I$, the CMC $\psigma_{I'}$ is determined by the space $\mathrm{Ker}(\lambda_{I'})$, and $\psigma_{I'}\simeq \mathrm{Conv}\{N_i, i\in {I'}\}^{}\subset \mathrm{End}(V_{b', \bQ})$. The set-theoretic inclusion $\{\epsilon_i, i\in {I'}\}\hookrightarrow \{\epsilon_i, i\in I\}$ induces $\bQ^{|{I'}|}\hookrightarrow \bQ^{|I|}$, and $\mathrm{Ker}(\lambda_{I'})=\mathrm{Ker}(\lambda_I)\cap \bQ^{|{I'}|}$. We have the commutative diagram:
\begin{equation}\label{eqn:adjacencymapCMC}
\begin{tikzcd}
\bQ_{>0}^{|{I'}|} \arrow[d, hook] \arrow[r, hook, ""] & \bQ^{|{I'}|} \arrow[d, hook]\arrow[r] & \bQ^{|{I'}|}/\mathrm{Ker}(\lambda_{I'}) \arrow[d, hook] \\
\bQ_{>0}^{|I|}  \arrow[r, hook, ""] & \bQ^{|I|}\arrow[r] & \bQ^{|I|}/\mathrm{Ker}(\lambda_I) 
\end{tikzcd}
\end{equation}
\begin{definition}
   The map $\bQ^{|{I'}|}/\mathrm{Ker}(\lambda_J)\hookrightarrow \bQ^{|I|}/\mathrm{Ker}(\lambda_I)$ in the diagram \eqref{eqn:adjacencymapCMC} induces a canonical map $\xi_{I,I'}: \psigma_{I'}\hookrightarrow \psigma_I$ which we will call the adjacency map.
\end{definition}
 
%--------------------------------
\subsection{Definition of CMCX}
%--------------------------------

\begin{definition}\label{def:CMCX}
    A combinatorial monodromy complex (CMCX) $\mathcal{T}$ associated to $\bV\rightarrow B$ and $\hat{B}$ is defined by assigning each $I\in \mathcal{I}$ with $Z_I\neq \emptyset$ a polyhedral complex $\Sigma_{\mathcal{T}}(I)$ supported on the closed combinatorial monodromy cone $\psigma_I\subset \bQ^{|I|}/\mathrm{Ker}(\lambda_I)$, such that for any $I'\leq I$, the adjacency map $\xi_{I,I'}$ induces an embedding of subcomplexes $\Sigma_{\mathcal{T}}(I')\leq \Sigma_{\mathcal{T}}(I)$.
\end{definition}

One immediate example in which a CMCX exists and can be understood in the simplest way is:
\begin{example}\label{Example:stronglysimplicial}
    In the case when $\mathrm{dim}(\psigma_I)=\mathrm{codim}_{\hat{B}}(Z_I)$ for every $I\subset \mathcal{I}$ with $Z_I\neq \emptyset$, there is a canonical minimal CMCX $\mathcal{T}_0$ defined by associating each $I\subset \mathcal{I}$ the complex $\{\psigma_I \ \text{and all of its proper faces}\}$.
\end{example}

More generally, we may find the following cases in which a natural CMCX could be found. 

\begin{definition}\label{Def:simplicialvhs}
    We say the completion $\hat{B}$ of $B$ is simplicial corresponding to the $\bZ$-PVHS $\bV\rightarrow B$ or just simplicial, if for any $b\in Z_I^*\subset \hat{B}$ the CMC $\psigma_b$ is a simplicial cone, and for any $i\in I$ the local monodromy logarithm $N_i$ (view as a CMC) is a face of $\psigma_b$ via the adjacency map $\xi_{I,i}$. We say  $\bV\rightarrow B$ has simplicial local monodromy or just say  $\bV\rightarrow B$ is simplicial if such a completion $\hat{B}$ of $B$ exists.
\end{definition}

\begin{theorem}\label{Thm:simplicialCMCX}
If $\hat{B}$ is simplicial corresponds to $\bV\rightarrow B$, then the collection of cones $\mathcal{T}_0$ associating each strata $Z_I^*\subset \hat{B}-B$ the polyhedral complex $\{\text{All faces of }\sigma_I\}$ is a CMCX. In particular, this includes the case when $B$ is a curve or a surface.
\end{theorem}
\begin{proof}
In these cases it is straight-forward to verify that every adjacency map $\xi_{I,I'}$ induces an embedding of sub-complexes.
\end{proof}

%--------------------------------
\subsection{CMCX and logarithmic modifications}
%--------------------------------

Note that a CMCX $\mathcal{T}$ depends on not only the original $\bZ$-PVHS $\bV\rightarrow B$ but also on a fixed completion $B\subset \hat{B}$. We may associate $\hat{B}$ the compactifying logarithmic structure corresponds to $(\hat{B},Z)$. Let $\olB\rightarrow \hat{B}$ be a logarithmic modification such that $\overline{Z}:=\olB-B=\bigcup_{J\in \hat{\cI}}Z_J$ is also an SNC divisor. Therefore, for any $b'\in \olB$, there is also an associated nilpotent orbit $(\sigma_{b'}, F_{b'})$ defined up to $\Gamma$ by Corollary \ref{cor:schmidNOT}.

\begin{lemma}
    If $b'\in Z_J^*\subset\olB$ lies above $b\in Z_I^{*}\subset\hat{B}$, then the map $\olB\rightarrow \hat{B}$ induces a map of CMC's $\psigma_J\rightarrow \psigma_I$, denoted as $\eta_{I,J}$ or $\eta_{b,b'}$.
\end{lemma}

\begin{proof}
    Any $1$-parameter neighborhood $b'\in \Delta_{b'}\subset \hat{B}$ with $\Delta_{b'}^*\subset B$ descends to a $1$-parameter neighborhood $b\in \Delta_{b}\subset \hat{B}$ with $\Delta_{b}^*\subset B$ via the map $\olB\rightarrow \hat{B}$. This gives an inclusion of some $N'\in \sigma_{b'}$ to some $N\in \sigma_b$. Running over all such $1$-parameter neighborhoods and look at the CMC, we have the desired inclusion $\eta_{J,I}$.
\end{proof}

\begin{definition}
    We refer to the map $\eta_{I,J}$ (or $\eta_{b,b'}$) as the descendant map.
\end{definition}

With the definition of the descendant map, one may define the compatibility of $\mathcal{T}$ and the logarithmic modification $\olB\rightarrow \hat{B}$ as follows.

\begin{definition}\label{def:CMCXcompatible}
     We say a CMCX $\mathcal{T}$ associated to $\olB$ is compatible with the logarithmic modification $\olB\rightarrow \hat{B}$ (or vice versa) if for any $b'\in\olB$ lying above $b\in \hat{B}$, there exists a unique $\ptau_b\in \Sigma_{\mathcal{T}}(b)$\footnote{For convention, by $\Sigma_{\mathcal{T}}(b)$ we mean $\Sigma_{\mathcal{T}}(I)$ for the strata $I$ containing $b$.} such that $\eta_{b',b}(\psigma_{b'}^{\circ})\subset \ptau_b^{\circ}$. Moreover, this unique $\ptau_b$ satisfies $(\tau_b, F_{b'})$ is a nilpotent orbit.
\end{definition}

\begin{remark}
    Starting from this point, When a descendant map $\eta_{I,J}$ or an adjacency map $\xi_{I,J}$ are clear, we ignore the map and denote the corresponding inclusions as $\psigma_J\hookrightarrow\psigma_I$ or $\psigma_J\subset \psigma_I$. 
\end{remark}

If we assume the global monodromy group $\Gamma$ is abelian, a CMCX is equivalent to a $\Gamma$-compatible fan $\Sigma$ given by \cite[Theorem 4.3.1(iii)]{KU08}. In general, the notion of a CMCX is different from the notation of fan as it depends more on the base of the map, and in particular does not require any kind of compatibility with global monodromy.

We breifly recall some of the literature of toroidal embeddings and toroidal/logarithmic modifications. Suppose a CMCX $\cT$ corresponds to the $\bZ$-PVHS $\bV\rightarrow B$ and the completion $B\subset \hat{B}$ with $\hat{B}-B$ an SNC divisor are given.  

 In the theory of \cite{KKMSD89}, the pair $(\hat{B}, B)$ is a \textit{toroidal embedding.} In particular, there exists a canonically-constructed toroidal fan $\pmb{C}$ associated to the toroidal embedding as \cite[pg. 71]{KKMSD89} which, in our case, is a topological polyhedral complex consisting of simplicial cones.

If $\pmb{C}'$ is another complex refining $\pmb{C}$, then there is another toroidal embedding pair $(\hat{B}', B)$ with toroidal fan $\pmb{C}'$ and a proper birational morphism $\hat{B}'\rightarrow \hat{B}$ restricting to the identity map on $B$. This is called a toroidal modification. As a summary, similar to the case for toric varieties, we have the correspondence:
\begin{center}
    \textit{\{Finite ``allowable" subdivisions of toroidal fans \}$\Leftrightarrow$ \{Toroidal modifications\},}
\end{center}
see \cite[pg. 90]{KKMSD89}. In our case, ``allowable" subdivisions are exactly finite rational polyhedral subdivisions.

The classical theory of toroidal embeddings and modifications gneeralize to the logarithmic setting, see \cite[Chap. 3.6]{KU08} for an overview. In particular, the term ``logarithmic modification" has the same meaning as ``toroidal modification" in our case, and we will use them interchangably.

With the general philosophy of logarithmic modifications and polyhedral subdivision of toroidal fans, we are ready to prove the first technical result of this paper below. Note that the proof uses definitions of semicomplex and the canonical complexification which are not defined until Section \ref{Sec08}.

%\jacob{Ok I buy this but I don't think it makes sense as written, and we need specific references. For example, $\hat{B}(I)$ doesn't make sense for an individual $I$, I think only the whole collection makes sense at once and the compatibility is essential for this. Anwho, this shouldn't be hard just need references for 1. Kato fans 2. Logarithmic modifications (I understand intuitively what you mean but the prcise notio is unclear and non-stndard) 3. This result that we can find a modification.}

\begin{theorem}\label{Thm:CMCXgiveslogmodification}
   Given a CMCX $\cT$ on $\hat{B}$, there exists a logarithmic modification $(\olB, B)$ of $(\hat{B}, B)$ compatible with $\cT$.
\end{theorem}

\begin{proof}
    Let $\pmb{C}_{\hat{B}}$ be the toroidal fan associated to the toroidal embedding $(\hat{B}, B)$. It associates to each $I\in \mathcal{I}$ a closed simplicial cone $\pmb{C}_I$ with dimension same as the codimension of $Z_I^*\subset\hat{B}$ because of the SNC property. There is a combinatorial monodromy map $\mathfrak{m}_I: \pmb{C}_I\rightarrow \psigma_I$ sending the $1$-dimensional face of $\pmb{C}_I$ corresponding to the divisor $Z_i$ to the monodromy logarithm $\pmb{N}_i\subset \psigma_I$ (viewed as a CMC), and extended linearly. 

    For each $I\in \mathcal{I}$, since $\Sigma_\cT(I)$ is a polyhedral decomposition of the CMC $\psigma_I$, its pre-image under $\mathfrak{m}_I$ gives a polyhedral decomposition of $\pmb{C}_I$ which we denote as $\pmb{C}_{\cT}(I)$. By Definition \ref{Def:semicomplex} and Lemma \ref{Lem:StandardSemiComp}, the collection obtained from $\pmb{C}_{\hat{B}}$ by replacing every $\pmb{C}_K\leq \pmb{C}_I$ with the set $\pmb{C}_{\cT}(I)|_{\pmb{C}_K}$ is a semicomplex. Denote its canonical complexification by Definition \ref{Def:CanonicalComp} as $\pmb{C}_I'$. Since $\pmb{C}_I'$  is a refinement of $\pmb{C}_I$, it gives a logarithmic modification $\hat{B}(I)\rightarrow \hat{B}$. Let $\pmb{C}_{\olB}$ be a common refinement of all $\pmb{C}_I'$ and $\olB\rightarrow \hat{B}$ be the corresponding logarithmic modification. By further resolving $\olB$ if necessary, we may assume $\olB$ is smooth and $\olB-B$ is an SNC divisor.
    
     We show the logarithmic modification $\olB\rightarrow \hat{B}$ has the required properties. For any $b'\in Z_{J}^*\subset \olB$ lying above $b\in Z_I^*\subset \hat{B}$, by the construction of $\hat{B}$ there exists a unique cone $\pmb{c}_I\in \pmb{C}'_I$ such that locally around $b\in \hat{B}$, $b'$ lies above the boundary component of $\hat{B}(I)\rightarrow \hat{B}$ determined by $\pmb{c}_I$. The cone $\pmb{c}_I$ is mapped into a CMC $\ptau_{b'}\in \Sigma_{\cT}(I)$, and it is clear that $(\tau_{b'}, F_{b'})$ must be a nilpotent orbit as $(\sigma_I, F_{b'})$ is. 
\end{proof}

In the cases discussed by Theorem \ref{Thm:simplicialCMCX} where we have a canonical CMCX $\cT_0$ exists, the completion $\hat{B}$ of $B$ may be chosen in a simple way:

\begin{theorem}
    In the case $\hat{B}$ is simplicial corresponds to $\bV\rightarrow B$, $\olB=\hat{B}$ viewed as the trivial logarithmic modification of $(\hat{B}, B)$ is compatible with the CMCX $\cT_0$ defined by Theorem \ref{Thm:simplicialCMCX}.
\end{theorem}

\begin{proof}
    Directly follows from the definition.
\end{proof}

% ---------------------------------------
\section{Local monodromy and proper Kato--Nakayama--Usui maps}\label{Sec05}
% ---------------------------------------

During this section we fix a logarithmic modification $\overline B\rightarrow \hat{B}$ as well as a CMCX $\mathcal{T}$ compatible with $\olB$. We define $\overline{\mathcal{I}}$ is the collection of boundary strata of $\overline{B}$. Write $\olB-B=\bigcup_{J\subset \overline{\cI}}Z_J$ as the union of its smooth boundary strata, and $Z_J^{*}:=Z_J-\cup_{J\not \subset K}Z_K$.

%--------------------------------------------------
\subsection{The Nilpotent Cone Closure (NCC)}
%--------------------------------------------------

 Let $J\subset \overline{\cI}$ be a subset such that $Z_J\neq \emptyset$. For any point $y\in Z_J^{*}$ lying above some $b\in Z_I^{*}\subset \hat{B}$, by Definition \ref{def:CMCXcompatible} there exists a unique $\ptau_J\in \Sigma_{\cT}(I)$ such that $\psigma_y=\psigma_J\subset \ptau_J^{\circ}$. 

%Based on these data, we will construct the \textbf{nilpotent cone closure} $\cZ_J\subset \olB$ of $Z_J$.

We first define a space $\cZ_J^0\subset \olB-B$ as a union of strata in $\olB-B$. Suppose $K\subset \overline{\cI}$ such that $Z_{J\cup K}\neq \emptyset$. For any point $y'\in Z_K^{*}$, we have $\psigma_{y'}=\psigma_K\subset \ptau_{K}\in \Sigma_{\cT}(I')$ for some $I'\subset \cI$. Similarly for any $y''\in Z_{J\cup K}^{*}$ we may construct $\psigma_{y''}=\psigma_{J\cup K}\subset\ptau_{J\cap K}\in \Sigma_{\cT}(I'')$. Since $I'\subset I''$ implies $\Sigma_{\cT}(I')$ is a sub-complex of $\Sigma_{\cT}(I'')$, we have $\ptau_J$ and $\ptau_K$ are faces of $\ptau_{J\cup K}$. 

\begin{definition}
    We say two subsets $J, K\subset \cI$ are commensurable if $Z_{J\cup K}\neq \emptyset$, $\bQ\psigma_J\subset\bQ\ptau_K$, and $\bQ\psigma_K\subset\bQ\ptau_J$, all regarded as subspaces of $\bQ\ptau_{J\cup K}$. In this case we denote it as $J\sim K$.
\end{definition}

\begin{lemma}\label{Lem:CommStrataNCC}
    $J\sim K$ if and only if $Z_{J\cup K}\neq \emptyset$ and $\ptau_J=\ptau_K=\ptau_{J\cup K}$, in which case $J\sim J\cup K$ and $K\sim J\cup K$.
\end{lemma}

\begin{proof}
    Since $\ptau_J$ is an open neighborhood of $\psigma_J$ in $\bQ\ptau_J$ and $\psigma_K\subset \bQ\ptau_J$, we must have $\mathrm{Conv}\{\psigma_J, \psigma_K\}=\psigma_{J\cup K}\subset \bQ\ptau_J$ and has non-empty intersection with $\ptau_J$, hence $\ptau_J=\ptau_{J\cup K}$. Similarly $\psigma_{J\cup K}\subset \bQ\ptau_K$ and $\ptau_K=\ptau_{J\cup K}$, hence $J\sim J\cup K$ and $K\sim J\cup K$.
\end{proof}

\begin{remark}
    "$\sim$" is in general not an equivalence relation for the following reasons: 
    \begin{enumerate}
        \item Given $I\sim J$ and $J\sim K$, we may have $Z_I\cap Z_K=\emptyset$
        \item Even if $Z_I\cap Z_K\neq \emptyset$, the CMC's $\psigma_{I\cup J}, \psigma_{J\cup K}, \psigma_{I\cup K}$ could have different dimensions.
    \end{enumerate}
\end{remark}

\begin{definition}\label{Def:Nilpconeclosure}
 For any  $J\subset \overline{\cI}$, the \textit{nilpotent cone closure (NCC)} of $J$ relative to $\mathcal{T}$, denoted as $\fC_{\cT}(J)$ or simply just $\fC(J)$ when $\cT$ is fixed, is the equivalence class of $J$ for the equivalence relation generated by $\sim$.
  
\end{definition}
We denote $Z_{\fC(J)}:=\bigcup_{K\in \fC(J)}Z_K^{*}$.

\begin{example}
    In the case given by Example \ref{Example:stronglysimplicial}, we have $\fC_{\cT_0}(J)=:\fC(J)=\{J\}$ for every $J\in \overline{\cI}$ and $Z_J\neq \emptyset$.
\end{example}

%--------------------------------------------------
\subsection{Connected Limiting Fiber (CLF)}
%--------------------------------------------------

 For any index $K\subset\overline{\cI}$ and $p\in Z_K^{*}$, we may choose a neighborhood $p\in B_p\subset \olB$ such that $B_p\cap B$ is homeomorphic to a product of (punctured and unpunctured) disks. Let  $\cU_p:=\cU_p^-\cap B$.
Fix some $J\subset \overline{\cI}$. Let $\cU_{\fC(J)}\subset B$ be the punctured neighborhood of $Z_{\fC(J)}$ obtained by
\[
\cU_{\fC(J)}:=\bigcup_{p\in Z_{\fC(J)}}\cU_p, \ \
B_{\fC(J)}:=\bigcup_{p\in Z_{\fC(J)}}B_p
\]
where for each $p$, $\cU_p$ is a neighborhood supporting the existence of diagram \eqref{eqn:localperiodmapandlocalKNU}. Moreover, by shrinking the $\cU_p$ we may insist that $B_{\fC(J)}$ deformation retracts to $Z_{\fC(J)}$.

Let $\Gamma_{\fC(J)}$ be the monodromy group of $\bV|_{\cU_{\fC(J)}}\rightarrow \cU_{\fC(J)}$.
\begin{lemma}\label{Lem:localmonodromyCMCX}
    $\Gamma_{\fC(J)}$ stabilizes $\tau_J\subset \mathfrak{g}_\bQ$, and centralizes $\tau_J\subset \mathfrak{g}_\bQ$ with neat $\Gamma$.
\end{lemma}

\begin{proof}
    Let $p\in Z_K^*$ for some $K\in\fC(J)$, and let $q\in U_p$. By construction we may assign to such a $q$ a cone $\tau_q\subset \End(\bV_q)$ corresponding to the monodromy cone around $p$, and its image in $\hat{B}$. Moreover, by the definition of the $\sim$ relation this assignment glues up to a sub-sheaf of monoids of $\End(\bV)\mid_{U_{\fC(J)}}$. Hence, we learn that the monodromy $\Gamma_{\fC(J)}$ stabilizes $\tau_J$. 
    
    For the final claim, note that neat groups are torsion-free. Since automorphisms of polyhedral cones are finite, the claim follows. 
\end{proof}

\begin{corollary}\label{cor:KNUmapstrata}
There are well defined $KNU$ maps
\begin{equation}\label{eqn:localperiodmapandlocalKNU}
\begin{tikzcd}
\cU_{\fC(J)}\arrow[d] \arrow[r, "\Phi_{\fC(J)}"] & \Gamma_{\tau_J}\backslash D \arrow[d]\\
B_{\fC(J)} \arrow[r, "\Phi_{\fC(J)}^{\tKNU}"]  & \Gamma_{\fC(J)}\backslash D_{\tau_J}
\end{tikzcd}
\end{equation}
    
\end{corollary}

\begin{proof}
    This is an immediate application of Corollary \ref{cor:extendingperiodmapKNU}. Indeed, the fan consisting of $\tau$ and its faces is stabilized by $\Gamma_{\fC(J)}$ by Lemma \ref{Lem:localmonodromyCMCX}, so we get a weak fan strongly compaible with $\Gamma_{\fC(J)}$. Moreover the definition of CMCX implies that the variation $\bV$ restricted to $U_{\fC(J)}$ is compatible in the sense of Definition \ref{def:compatibilityweakfan}. Indeed, each relevant $\sigma^\circ$ is contained in a unique open  $\tau^\circ$, and this induces a nilpotent orbit. Hence the corollary applies and the claim is proven.
\end{proof}

\begin{definition}\label{Def:CLF}
    Let $A\subset \olB$ be an analytic subvariety. We say $A$ is a connected limiting fiber (CLF) if for some fixed $J\subset \overline{\cI}$ and its associated $\Phi_{\fC(J)}^{\tKNU}$, $A\subset Z_{\fC(J)}$ is a connected component of a fiber of $\Phi_{\fC(J)}^{\tKNU}$. 
\end{definition}
% Let  $A={\Phi_{\fC(J)}^{\tKNU}}^{-1}(w)$ and we take a small open ball $U$ around $w$, which injects into lifts to $\Gamma_{\tau_J}\bs D_{\tau_J}$ . Then setting $B_A:={\Phi_{\fC(J)}^{\tKNU}}^{-1}(U)$ and $\cU_A:=B_A\cap B$ we see that  such that  $\cU_A$ has  monodromy $\Gamma_A\subset \Gamma_{\tau_J}$. 

\begin{lemma}\label{lem:KNUproper}
The CLFs are compact.    
\end{lemma}

\begin{proof}
 Suppose that $A = \Phi_{\fC(J)}^{-1}(w)$ is a CLF. If $A$ is not closed then there exists a point $q\in\ol{A}$ such that $q\in Z_K^*$ and $Z_K\subset \ol{Z_I}$ for some $I\in \fC(J)$. Wlog we may take $I=J$. It follows by construction of the CMCX that $\tau_J$ is a face of $\tau_K$.   Now consider the  map $\Phi_q^{\tKNU}: B_q\to \Gamma_{\tau_K}\bs D_{\tau_K}.$

Let $\Gamma^1_J\subset \bfG(\bZ)$ be the arithmetic and centralizing $\sigma_J$ (and therefore preserving $W(\sigma_J)$). Note that $\Gamma_{\fC(J)}\subset \Gamma^1_J$. By \cite[Thm 6.4]{BBKT24} there is a natural $\bR_{\alg}$-definable structure on $\Gamma^1_J\bs A(\sigma_J)$ such that $\Phi_J:Z_{\fC(J)}\to \Gamma^1_J\bs A(\sigma)$ is definable. Now $A$ is a connected component of a fiber of $\Phi_J$, and therefore $A$ must be $\bR_{\an,\exp}$-definable.

As a consequence $A\cap B_q$ has finitely many connected components, and let $A_0$ be one that contains $q$ in its closure. 

Now, consider by Corollary \ref{cor:extendingperiodmapKNU} the continuous morphism $f:B_p\to \Gamma_{\tau_K}\bs D_{\tau_K}$. 
By construction, and Theorem \ref{Thm:kumainthmA} (2), it follows that $f(A_0)=w$ is a single point.
By continuity,  the set $f^{-1}(b)$ is closed, and so contains $q\in \overline{A_0}$, and hence $f(p)=f(A_0)\subset \Gamma_{\tau_J}\bs A(\sigma_J)$, which is a contradiction.
\end{proof}

% ---------------------------------------
\subsection{Construction of the analytic structure}
% ---------------------------------------
We define $\ol{\wp}$ to be the space of CLFs. There is a natural topology on $\ol{\wp}$ induced by the quotient map $\ol{\Phi}:\ol{B}\to \ol{\wp}$, which is trivially an extension of $\Phi$. 

Our goal this subsection is to prove the following theorem:

\begin{theorem}\label{thm:mainstein}
The ringed space $(\ol{\wp},\ol{\Phi}_*\cO^{\an}_{\ol{B}})$ is a normal, compact, holomorphic variety. 
\end{theorem}

The rest of the section is devoted to the proof of Theorem \ref{thm:mainstein}, which clearly implies Theorem \ref{Thm:MainTHMformal}.

\begin{proof}

First, we construct a suitable partition into subsets foliated by CLFs:

\begin{lemma}\label{lem:CLFPartition}
For each CLF $A\subset Z_{\fC(J)}$, there exists 
an open set $A\subset U\subset \cU_{\fC(J)}$
such that $\Phi^{\tKNU}_{\fC(J)}\!\!\mid_U:U\to\Phi^{\tKNU}_{\fC(J)}(U)$ is a proper map. In particular, $U$ is foliated by CLFs. 
\end{lemma}

\begin{proof}

Let $w=\Phi^{\tKNU}_{\fC(J)}(A)$. First, since $A$ is compact we may and do shrink $\cU_A$ so that $A=(\Phi^{\tKNU}_{\fC(J)})^{-1}(w)$. 

Next, let $U'\subset \cU_A$ be an open set such that $\ol{U'}\subset \cU_A$, and let $W':=\Phi^{\tKNU}_{\fC(J)}(\ol{U'}\bs U')$. Now $W$ is compact  since it is the image of a  compact set, and since it is Hausdorff it is also closed. Moreover,$w\notin W'$.

Next, $w\in W$ be an open neighborhood such that $\ol{W}\cap W'=\emptyset$. Then $$(\Phi^{\tKNU}_{\fC(J)})^{-1}(\ol{W})\cap \ol{U'} = (\Phi^{\tKNU}_{\fC(J)})^{-1}(\ol{W})\cap U'.$$ We now set $U:= (\Phi^{\tKNU}_{\fC(J)})^{-1}(\ol{W})\cap U'$ and claim that this satisfies the assumption. 

Indeed, the map 
$\ol{U'}\to \Phi^{\tKNU}_{\fC(J)}(\ol{U'})$ is a proper map, and $\Phi^{\tKNU}_{\fC(J)}\!\!\mid_U:U\to\Phi^{\tKNU}_{\fC(J)}(U)$ is simply the base change under the inclusion $\Phi^{\tKNU}_{\fC(J)}(U)\to \Phi^{\tKNU}_{\fC(J)}(\ol{U'})$. The claim follows.
\end{proof}

For each $A$, let $U_A$ denote a neighborhood such as the above. By Theorem \ref{Thm:analyticimagepropermap} we see that $U_A\to \Phi^{\tKNU}_{\fC(J)}(U_A)$ is a proper map of analytic varieties. Let $U_A\to W_A$ be its Stein factorization. Now define $\Phi_A$ to be the restriction of $\ol{\Phi}$ to $U_A$. 
Then $\Phi_A:U_A\to \ol{\Phi}(U_A)$ agrees set theoretically with $U_A\to W_A$. Since the topology on $W_A$ is the quotient topology from $U_A$, and its structure sheaf $\cO^{\an}_{W_A}$ agrees with the pushforward of the structure sheaf of $U_A$, the Theorem follows.
\end{proof}

% ---------------------------------------
\section{Examples of projective completions in dimension $2$}\label{Sec06}
% ---------------------------------------

The purpose of this section is to provide some examples in the surface case supporting Conjecture \ref{Conj:ProjectiveCompletion}.

Let us assume $S$ is a quasi-projective surface. The $\bZ$-PVHS $\bV\rightarrow S$ gives the period map $\Phi: S\rightarrow \Gamma\backslash D$. We fix a normal projective completion $\hat{S}$ of $S$ which we do not assume to be a SNC completion. Let $S^{\circ}\subset \hat{S}$ be the complement of codimension-$2$ strata. The main result of this section is:

\begin{prop}\label{Prop:ProjectiveCrit}
    Let $\ol{S}\to \hat{S}$ be a birational morphism, such that $\ol{S}$ is a smooth SNC completion of $S$, and $S^\circ\subset \ol{S}$. Assume that every CLF in $\ol{S}$ maps to a point in $\hat{S}$. Then the the completed period map $\overline{\Phi}:\olS\rightarrow \overline{\wp}$ constructed by Theorem \ref{Thm:MainTHMformal} has the property that $\ol{\wp}$ is projective.  
    % If for some (hence any) proper birational morphism $\olS\rightarrow \hat{S}$ restricting to the identity on $S^\circ$, every CLF (Definition \ref{Def:CLF}) of $\olS$ is mapped to a point in $\hat{S}$, then we may choose the completed period map $\overline{\Phi}:\olS\rightarrow \overline{\wp}$ constructed by Theorem \ref{Thm:MainTHMformal} to lie in the category of projective varieties.
\end{prop}

\begin{proof}
    Through the blow-down map $\pi:\olS\rightarrow \hat{S}$, any positive-dimensional CLF in $\olS$ must be mapped to a codimensional-$2$ point in $\hat{S}$ according to the assumption. 
    Therefore, since the completed period map $\olphi: \olS\rightarrow \overline{\wp}$ has connected fibers, we
    see that $\pi$ factors as
    $\olS\to \ol{\wp}\to \hat{S}$. The claim now follow by \cite[Thm 2.6.2.]{Abr2002}.
    
\end{proof}

%This criteria is powerful in practice: It avoids the steps finding a lift with neat monodromy as well as an SNC completion of $B$. 

Proposition \ref{Prop:ProjectiveCrit} will be used to justify the projectivity of the extended period map in the following example.

\subsection{Hosono--Takagi's example} We look at the family of Calabi-Yau threefolds studied by \cite{HT14}, \Cite{HT18}, \cite{Den22} and \cite{DR23}. The period map used here is the one introduced by \cite{HT14} and \cite{HT18}, which is a period map of type $(1,2,2,1)$ as follows: 
\begin{equation}
    \Phi_0: \mathbb{P}^2\backslash \mathrm{Dis}=:S\rightarrow \Gamma\backslash D.
\end{equation}
The discriminant locus is given by:
\begin{equation}
    \mathrm{Dis}:= D_0\cup D_1\cup D_2\cup C
\end{equation}
where $D_i$ are coordinate divisors in $\mathbb{P}^2$, and $C$ is an irreducible quintic with $6$ self-intersection nodes and $1$ tangent point of order $5$ with each coordinate divisor. The picture of this base can be found at \cite[Fig. 6.1]{HT14}. The precise monodromy group $\Gamma$ is not known, but from \cite[Sec. 7]{Den22} we know the algebraic monodromy group $\overline{\Gamma}^{\mathbb{Q}}$ is $\mathrm{Sp}(6,\mathbb{Q})$. In \cite[Sec. 4]{HT18} it is shown that $\Gamma$ is not neat.

Let $\varphi:S'\rightarrow S$ be a finite et\'ale morphism such that the lift $\bV'\rightarrow S'$ of $\bV\rightarrow S$ has neat monodromy group $\Gamma'\leq \Gamma$. We may find a projective completion $\overline{\varphi}:\overline{S'}\rightarrow \bP^2$ of $\varphi$ which is a finite morphism.  This can be done by normalizing $\bP^2$ in the function field of $S'$. As a consequence, the $\bZ$-PVHS $\bV'\rightarrow S'$ and $S'\subset\overline{S'}$ satisfy the condition in Proposition \ref{Prop:ProjectiveCrit} if and only if $\bV\rightarrow S$ and $S\subset \bP^2$ do. 

By blowing up $\bP^2$ at each tangent point for $5$ times, we get a smooth projective completion $\olS$ of $S$ with SNC boundary divisors.
% take the strict transform of $B_0$, and lift $B_0$ to some finite etale cover $B$ to obtain neat monodromy $\Gamma\leq \Gamma_0$, we get a lifted period map:
% \begin{equation}\label{Fig:liftbyneatsubgroup}
% \begin{tikzcd}
% B\arrow[d] \arrow[r, "\Phi"] & \Gamma\backslash D \arrow[d] \\
% B_0 \arrow[r, "\Phi_0"] & \Gamma_0\backslash D
% \end{tikzcd}
% \end{equation}
% where the base $\widetilde{\mathbb{P}^2\backslash \mathrm{Dis}}=:S$ is quasi-projective with normal-crossing boundary divisors. 
We mark all local monodromy nilpotent cones as follows:
\begin{enumerate}\label{alllocalmonodromynilporbits}
    \item Let $\sigma_x, \sigma_y, \sigma_z$ be the ($2$-dimensional) local monodromy nilpotent cones around the intersection of each pair of coordinate divisors;
    \item Let $\sigma_0, \sigma_1, \sigma_2$ be the local monodromy nilpotent cones obtained by blowing-up each of the fifth tangent point. As \cite[Sec. 6.1]{Den22} shows, the blow-up process annihilates the order-$5$ semisimple part of the monodromy operators around $C$ and leave the unipotent part invariant;
    \item Let $\tau_j, \ j=1,...,6$ be the ($2$-dimensional) local monodromy nilpotent cones obtained from all self-intersection nodes of $C$.
    \item Let $\mathcal{E}$ be the remaining cones around the exceptional divisors.
\end{enumerate}
By computational results in \cite{HT14} and \cite{HT18}, all $2$-dimensional local monodromy nilpotent orbits and their Hodge degeneration types are listed as follows. Here we use the LMHS types defined by \cite[Example 5.8]{KPR19}.
\begin{enumerate}\label{typeofalllocalmonodromy}
    \item $\sigma_x, \sigma_y, \sigma_z$ are all of the type $\langle \mathrm{IV}_2|\mathrm{IV}_2|\mathrm{IV}_2\rangle$;
    \item $\sigma_0, \sigma_1, \sigma_2$ are all of the type
    $\langle\mathrm{IV}_2|\mathrm{IV}_2|\mathrm{I}_1\rangle$;
    \item $\tau_j, \ 1\leq j\leq 6$ are all of the type
    $\langle\mathrm{I}_1|\mathrm{I}_2|\mathrm{I}_1\rangle$.
    \item Any ($1$-dimensional) cone in $\mathcal{E}$ is of the type 
    $\langle\mathrm{I}_1|\mathrm{I}_1|\mathrm{I}_1\rangle$.
\end{enumerate}

\begin{lemma}
    In $\olS$ the only possible positive-dimensional CLF are the exceptional divisors of $\olS\rightarrow \bP^2$. 
\end{lemma}
\begin{proof}
    There are no curves in $S$ contracted by $\Phi_0$ as in $\bP^2$ the closure of any such curve must intersect the boundary divisor $\mathrm{Dis}$ while no points on $\mathrm{Dis}$ give the trivial monodromy. The fact that nilpotent cones of type (1) and (3) listed above are of two-dimensional implies non of $D_0, D_1, D_2$ or $C$ can be a $1$-dimensional CLF.  
\end{proof}

As a consequence, if we choose $\bP^2$ instead of $\olS$ to be the projective completion of $S$, the assumption of Proposition \ref{Prop:ProjectiveCrit} is met and we may conclude the period map $\Phi_0$ (after being lifted to the finite et\'ale cover $S'\rightarrow S$) admits a projective completion by Theorem \ref{Thm:MainTHMformal}.

\subsection{More examples}

In the paper \cite{CD24}, Chongyao Chen and the first author have investigated several $2$-parameter families of Calabi-Yau threefolds defined over normal rational surfaces as well as their period mappings. The LMHS types along each boundary strata are analyzed in detailed in \cite[Sec. 4]{CD24}. We claim that all of these families satisfy the criteria given by Proposition \ref{Prop:ProjectiveCrit}, therefore the corresponding period mappings admit projective completions. We left the details to readers.

\section{Algebraicity of weakly special subvarieties}\label{Sec07}
% ---------------------------------------

In this section we introduce another application of our main theorem. It is well-known that weakly special subvarieties associated to the period mapping $\Phi: B\rightarrow \Gamma\backslash D$ are algebraic subvarieties of $B$. The original version of this theorem is due to \cite{CDK95} and in this form it was proven in \cite{BKT20}, \cite{BBT22} using  o-minimal geometry and the definability of period mappings, though it follows formally from \cite{CDK95}. We will give yet another proof of this theorem using log-manifolds. 

Regarding the necessary backgrounds, we refer readers to \cite{KO21} or \cite{Kl22} for a comprehensive overview. 

\subsection{Weakly special subvarieties}

Let $(\mathbf{G}, D)$ be the Hodge datum associated to the $\bZ$-PVHS $\bV\rightarrow B$ and the period map $\Phi: B\rightarrow \Gamma\backslash D$. A weak Hodge datum $(\mathbf{H}, M)$ is given by a weak Mumford-Tate subdomain $M\leq D$ and its automorphism group $\mathbf{H}\leq \mathbf{G}$ defined over $\bQ$. Let $\Gamma_H:=\Gamma\cap H:=\mathbf{H_\bQ}$, we have a natural embedding of Hodge varieties $\Gamma_H\backslash M\subset \Gamma\backslash D$.

\begin{definition}
    An irreducible analytic component of $\Phi^{-1}(\Gamma_H\backslash M)$ in $B$ is called a weakly special subvariety with Hodge datum $(\mathbf{H}, M)$.
\end{definition}

As mentioned above, the following is a well-known result in Hodge theory and functional transcendence. It is first proved by \cite{CDK95} for the case of special subvarieties, then generalized to the weakly case.

\begin{theorem}\cite[Prop. 4.8]{KO21}\label{Thm:weaklyspecialalgebraic}
   $\Phi^{-1}(\Gamma_H\backslash M)$ is an algebraic subvariety of $B$. In particular, weakly special subvarieties of $B$ are algebraic. 
\end{theorem}

The proof uses the definability of period mappings \cite{BKT20} and o-minimal geometry \cite{BBT22}. We will give another proof of this result using the generalized toroidal completion given by Theorem \ref{Thm:MainTHMformal}.  The main idea is to extend the analytic condition on $B$ defining the weakly special subvarieties to $\olB$ over which a extended period map provided by Theorem \ref{Thm:MainTHMformal} exists. This may be seen as a generalization of \cite{CDK95} using the language of logarithmic geometry.

\subsection{The period torsors}

We first introduce the construction of period torsors following \cite[Chap. 3]{KU08}. Fix a monodromy cone $\sigma\subset \mathfrak{g}_{\bQ}$ whose integral monoidal structure is given by $\exp(\sigma_{\bC})\cap G_{\bZ}$. Let $\Gamma(\sigma)\leq \Gamma$ be the submonoid consists of elements whose logarithms lie in $\sigma$.  Let $\mathrm{toric}_{\sigma}$ be the affine toric variety associated with $\sigma$ and $\mathrm{torus}_{\sigma}$ be the torus embedding. We have the natural identification:
\begin{align}
    &\mathrm{toric}_{\sigma}\simeq \mathrm{Hom}(\Gamma(\sigma)^{\vee}, \bC),\\
    &\mathrm{torus}_{\sigma}\simeq \mathrm{Hom}(\Gamma(\sigma)^{\vee \mathrm{gp}}, \bC^{*})\simeq C^{*}\otimes \Gamma_{\sigma}
\end{align}
where the dual and ``Hom" are taken in the category of multiplicative monoids. Define
\begin{align}
\check{E}_{\sigma}&:=\mathrm{toric}_{\sigma}\times \check{D}\\
\nonumber \tilde{E}_{\sigma}&:=\{(q,F)\in \check{E}_{\sigma} \ | \ NF^{\bullet}\subset F^{\bullet-1} \ \forall  N\in \sigma(q)\}\\
\nonumber E_{\sigma}&:=\{(q,F^{})\in \tilde{E}_{\sigma} \ | \ (\sigma(q), \mathrm{exp}(\mathrm{log}_{\Gamma_{\sigma}}q)F) \hbox{ is a } \sigma(q)-\hbox{nilpotent orbit } \}
\end{align}
Here $\sigma(q)\leq \sigma$ is the face corresponds to the torus orbit containing $q$.

Following Kato--Usui's book \cite{KU08}, we will see these spaces carry rich structures in the theory of logarithmic geometry. Regarding basic definitions and examples in logarithmic geometry including logarithmic analytic spaces, logarithmic differential forms, and the strong topology, we refer readers to \cite[Chap. 2,3]{KU08}. Here we introduce one of the most important definitions.

\begin{definition}\cite[Def. 3.5.7]{KU08}
    By a logarithmic manifold $X$, we mean a $\bC$-logarithmic local ringed space with an open cover $\{U_\lambda\}_{\lambda\in \Lambda}$ such that for each $\lambda\in \Lambda$ there exists a logarithmically smooth fs logarithmic analytic space $Z_\lambda$ and a finite subset $I_\lambda$ of logarithmic differential $1$-forms on $Z_\lambda$, such that
    \[
    U_{\lambda}\simeq \{z\in Z_\lambda, \ I_{\lambda,z}=0\}
    \]
    where the RHS is endowed with the strong topology in $Z_\lambda$ and the pull back of the logarithmic structure on $Z_\lambda$. In this case, we call $(Z_\lambda, I_\lambda)$ a chart of $X$.
\end{definition}

%By the definition of logarithmic manifolds \cite[Definition 3.5.7]{KU08}, $\tilde{E}_{\sigma}$ and $E_{\sigma}$ is locally charted by the vanishing loci of some finitely-generated ideal $\mathcal{L}$ of the sheaf of logarithmic $1$-forms $\omega^1_{\check{E}_{\sigma}}$ over $\check{E}_{\sigma}$. 

By \cite[Chap. 3]{KU08}, $\check{E}_{\sigma}$ admits a natural fs logarithmic analytic structure from $\mathrm{toric}_{\sigma}$ via the projection map, and $\tilde{E}_{\sigma}$ and $E_{\sigma}$ admits the structure of logarithmic manifolds. Their local charts are described as follows: Let $\cF^{\bullet}$ be the universal Hodge filtration of $\mathcal{O}_{\check{D}}\otimes_{\bZ}V_{\bC}$ over $\check{D}$, $\mathcal{G}:=\mathcal{O}_{\check{D}}\otimes_{\bC}\sigma_{\bC}$ and $\mathcal{G}^{-1}:=\{X\in \mathcal{G}, \ X\cF^{\bullet}\subset \cF^{\bullet-1}\}$. Let $\mathcal{P}$ be the pullback of $\cG/\cG^{-1}$ to $\tilde{E}_{\sigma}$ via the projection map. There is a canonical projection morphism $\eta: \mathcal{O}_{\check{E}_{\sigma}}\otimes_{\bC} \sigma_{\bC}\rightarrow \cP$ giving a global section of $(\mathcal{O}_{\check{E}_{\sigma}}\otimes_{\bC} \sigma_{\bC})^{\vee}\otimes \cP$ which is isomorphic to $\omega^1_{\check{E}_{\sigma}}\otimes \cP$. Therefore, the ideal $\mathcal{L}\subset \omega^1_{\check{E}_{\sigma}}$ is generated by the coordinates of $\eta$ under local frames of $\cP$.

Consider the map:
\begin{equation}\label{Eqn:sigmatorsormap}
    \Theta_{\sigma}: E_{\sigma}\rightarrow \Gamma_{\sigma}\backslash D_{\sigma}, \ (q, F)\rightarrow (\sigma(q), \mathrm{exp}(\mathrm{log}_{\Gamma_{\sigma}}q)F)
\end{equation}
 which is the quotient of the action:
\begin{equation}\label{Eqn:torsoraction}
    a\in \sigma_{\mathbb{C}}: E_{\sigma}\rightarrow E_{\sigma}, \ (a, (q,F))\rightarrow (\mathbf{e}(a)q, \exp(-a)F)
\end{equation}
where the map $\mathbf{e}(a)$ is defined by
\begin{equation}
    \mathbf{e}(z\log(\gamma))=\exp(2\pi iz)\otimes \gamma\in \mathrm{torus}_{\sigma}.
\end{equation}
This realizes $E_{\sigma}$ as a $\sigma_{\mathbb{C}}$-torsor over $\Gamma_{\sigma}\backslash D_{\sigma}$ in the category of logarithmic manifolds by \cite[Section 7.2]{KU08}.

According to \cite[Section 7.3.5]{KU08}, for $x:=(q, F)\in E_{\sigma}\subset \check{E}_{\sigma}$, the structure of logarithmic manifold on $\Gamma_{\sigma}\backslash D_{\sigma}$ around $\Theta_{\sigma}(x)$ is described as follows: Let $\check{A}_x\subset \check{E}_{\sigma}$ be an analytic variety through $x$ such that:
\begin{enumerate}
    \item The projection of $\check{A}_x$ to $\mathrm{toric_{\sigma}}$ is an open subvariety $U_0$ of $\mathrm{toric}_{\sigma(q)}$.
    \item The projection of $\check{A}_x$ to $\check{D}$ is a subvariety $U_1$ of $\check{D}$ intersecting $\exp(\sigma(q)_{\bC})F$ transversely at $F$, and $\mathrm{dim}(U_1)=\mathrm{dim}(\check{D})-\mathrm{dim}(\sigma(q))$.
    \item $\check{A}_x\simeq U_0\times U_1$.
\end{enumerate}
Let $A_x:=\check{A}_x\cap E_{\sigma}$, then $A_x\subset E_{\sigma}$ acquires a structure of logarithmic submanifold of $E_{\sigma}$, and by \cite[Lemma 7.3.3]{KU08}, We may associate the logarithmic manifold structure on $\Gamma_{\sigma}\backslash D_{\sigma}$ by regarding $A_x$ as a chart around $\Theta_{\sigma}(x)$.

\begin{remark}
    The construction of period torsor has been generalized to the Mumford-Tate domain setting in \cite{KP16} which is almost identical to the ambient period domain case in \cite{KU08}. We remark that similar construction also applies to the weak Mumford-Tate domain setting. 
\end{remark}

\subsection{The relative period torsors}

For the fixed weak Hodge datum $(\mathbf{H}, M)$ of $(\mathbf{G}, D)$. Denote $\mathfrak{h}:=\mathrm{Lie}(\mathbf{H})$ which is a Lie subalgebra of $\mathfrak{g}$, and $\sigma_{\fh}:=\sigma\cap\mathfrak{h}$. Define 
\begin{equation}
    \check{E}_{\sigma, M}:=\mathrm{toric}_{\sigma_{\fh}}\times \check{M}
\end{equation}
where $\check{M}$ is the compact dual of $M$. $\check{E}_{\sigma, M}$ may be realized as an algebraic subvariety of $\check{E}_{\sigma}$ via the embedding of cones $\sigma_{\fh}\hookrightarrow\sigma$. Let $\tilde{E}_{\sigma, M}:=\tilde{E}_{\sigma}\cap\check{E}_{\sigma, M}$ and $E_{\sigma, M}:=E_{\sigma}\cap\check{E}_{\sigma, M}$. Endow each space with the subspace topology. Endow $\check{E}_{\sigma, M}$ the fs logarithmic analytic structure from restricting the one on $\check{E}_{\sigma, M}$. This is different from the structure on $\check{E}_{\sigma, M}$ as an independent period torsor.

\begin{definition}
Let $X$ be a logarithmic manifold. We say $Y\subset X$ is a logarithmic submanifold if it is a logarithmic manifold as well as a $\bC$-logarithmic local ringed subspace of $X$, and for any $y\in Y\subset X$, there exists a local chart $(Z_\lambda, I_\lambda)$ of $X$ around $x$ such that there is a logarithmically smooth fs logarithmic analytic subspace $Z_{\lambda}'\leq Z_{\lambda}$, and $(Z_{\lambda}', I_{\lambda}|_{Z_{\lambda}'})$ is a local chart of $y$ in $Y$. 
\end{definition}

 Let $\Theta_{\sigma, M}$ be the restriction of $\Theta_{\sigma}$ on $E_{\sigma, M}$ and $\mathfrak{M}$ be its image. It is clear that $\tilde{E}_{\sigma, M}$ (resp. $E_{\sigma, M}$) is a logarithmic submanifold of $\tilde{E}_{\sigma}$ (resp. $E_{\sigma}$). We will show:

\begin{prop}\label{prop:logsubmanifold}
    $\mathfrak{M}$ is a logarithmic submanifold of $\Gamma_{\sigma}\backslash D_{\sigma}$, and the map $\Theta_{\sigma, M}: E_{\sigma, M}\rightarrow \mathfrak{M}$ is a $\sigma_{\fh, \bC}$-torsor of logarithmic manifolds.
\end{prop}
\begin{proof}
    Our proof relies on \cite[Lemma 7.3.3]{KU08}. To apply the lemma, we need to show $\sigma_{\fh, \bC}$ acts properly and freely on $E_{\sigma, M}$. Moreover, we need to show for any $x\in E_{\sigma, M}$, there exists a logarithmic submanifold $B_x\subset E_{\sigma, M}$ passing through $x$, a neighborhood $U$ of $0\in \sigma_{\fh, \bC}$ such that the group action map $U\times B_x\rightarrow E_{\sigma, M}$ is an isomorphism onto its image.

    We first show $\sigma_{\fh, \bC}$ acts freely on the fiber of $\Theta_{\sigma, M}$. From the result for $\Theta_{\sigma}: E_{\sigma}\rightarrow \Gamma_{\sigma}\backslash D_{\sigma}$, it is enough to show on the set-theoretical level that for $q\in \mathfrak{M}$, the fiber $\Theta^{-1}_{\sigma, M}(q)\subset E_{\sigma, M}$ is a $\sigma_{\fh, \bC}$-orbit. Fix $(q,F)\in E_{\sigma, M}$, for any $a\in \sigma_{\bC}$, by Equation \eqref{Eqn:torsoraction} it is enough to show $(\mathbf{e}(a)q, \exp(-a)F)\in E_{\sigma, M}$ if and only if $a\in \sigma_{\fh, \bC}$. For the first factor, note that $\mathbf{e}(a)q\in \mathrm{toric}_{\sigma_{\fh}}$ if and only if $\mathbf{e}(a)\in \mathrm{torus}_{\sigma_{\fh}}\simeq C^{*}\otimes \Gamma_{\sigma_{\fh}}$ if and only if $a\in \sigma_{\fh}$. For the second factor, note that the natural map $\sigma_{\bC}\rightarrow T_{\check{D}}F$ is an injection, hence $\exp(-a)F\in M$ if and only if $a\in T_{\check{M}}F\cap \sigma_{\bC}=\sigma_{\fh, \bC}$.
    
    Next we show the properness of the $\sigma_{\fh, \bC}$-action. By \cite[Lemma 7.2.7]{KU08}, it is enough to show for any directed families $\{(q_{\lambda}, F_{\lambda})\}_{\lambda\in \Lambda}\subset E_{\sigma, M}$ and $\{a_{\lambda}\}_{\lambda\in \Lambda}\subset \sigma_{\fh, \bC}$ such that both 
    $\{(q_{\lambda}, F_{\lambda})\}_{\lambda\in \Lambda}$ and $\{a_\lambda(q_{\lambda}, F_{\lambda})\}_{\lambda\in \Lambda}$ converge in $E_{\sigma, M}$, $\{a_{\lambda}\}$ must converge in $\sigma_{\fh, \bC}$. This immediately follows from the facts that $\{a_{\lambda}\}$ converges in $\sigma_{\bC}$ because of the proper action of $\sigma_{\bC}$ on $E_{\sigma}$, and $\sigma_{\fh, \bC}$ is a closed subspace of $\sigma_{\bC}$.

    Finally, similar to the ideas used in \cite[Section 7.3.5]{KU08}, we may choose $x=(q,F)\in \check{B}_x\subset \check{E}_{\sigma, M}$ such that $\check{B}_x\simeq B_1\times B_2$, where $B_1$ is an open neighborhood of $q$ in $\mathrm{toric}_{\sigma(q)\cap\fh}$, and $F\in B_2\subset M$ intersects $\exp((\sigma(q)\cap\fh)_{\bC})F$ transversely at $F$ and with complementary dimension in $\check{M}$. Let $B_x:=\tilde{B}_x\cap E_{\sigma, M}$. From the arguments we have in the last section for $\tilde{A}_x$ and $A_x$, this $B_x\subset E_{\sigma, M}$ satisfies the requirement. More precisely, the inclusions $\check{B}_x\subset\check{A}_x$ and $B_x\subset A_x$ realize $\mathfrak{M}$ as a logarithmic submanifold of $\Gamma_{\sigma}\backslash D_{\sigma}$. The proof is complete.
\end{proof}

\subsection{Proof of the algebraicity result}

We are ready to prove Theorem \ref{Thm:weaklyspecialalgebraic}. Fix a completion $\olphi: \olB\rightarrow \overline{\wp}$ given by Theorem \ref{Thm:MainTHMformal}. Once we can show any weakly special subvariety $S\subset B$ with Hodge datum $(\mathbf{H}, M)$ can be extended to an analytic subvariety of $\olB$, Serre's GAGA theorem will grant the algebraicity of $S$.

By Lemama \ref{lem:CLFPartition}, there is an 
open cover of  $\olB=\bigcup_kW_k$ such that each $W_k$ is foliated by CLF's and admits a proper KNU map $\overline{\Phi}_k: W_k\rightarrow \Gamma_{\tau_k}\backslash D_{\tau_k}$ for some monodromy cone $\tau_k\subset \mathfrak{g}$. Theorem \ref{Thm:analyticimagepropermap} implies each $\overline{\Phi}_k(W_k)$ is a complex analytic space.

Let $\mathfrak{M}_k$ be the image of $(\mathrm{toric}_{\tau_k\cap\fh}\times M)\cap E_{\tau_k, M}$ in $\Gamma_{\tau_k}\backslash D_{\tau_k}$ via the torsor map \eqref{Eqn:sigmatorsormap}. Let $x\in \mathfrak{M}_k\cap \overline{\Phi}_k(W_k)$. By Proposition \ref{prop:logsubmanifold}, we may choose neighborhoods $x\in A_x\subset \Gamma_{\tau_k}\backslash D_{\tau_k}$, $x\in V_x\subset \overline{\Phi}_k(W_k)$ and $B_x:=A_x\cap \mathfrak{M}_k$ such that $V_x\cap B_x$ is an analytic subspace of $V_x\cap A_x$. This implies for each $k$, $\olphi_k^{-1}(\mathfrak{M}_k)$ is an analytic subspace of $W_k$.

Finally, since each $\olphi_k$ has connected fibers, we may patch all $\olphi_k^{-1}(\mathfrak{M}_k)$ to an analytic subvariety $S$ of $\olB$ which must be algebraic by GAGA, and it is evident that $S\cap B=\Phi^{-1}(\Gamma_H\backslash M)$. The proof of Theorem \ref{Thm:weaklyspecialalgebraic} is complete. 
% ---------------------------------------
\section{On the existence of a CMCX}\label{Sec08}
%----------------------------------------

The purpose of this section is to prove Theorem \ref{Thm:CMCXexists} which will guarantee that Theorem \ref{Thm:MainTHMformal} is valid in the full generality. Some combinatorial definitions and results need to be established before the proof.

% ---------------------------------------
\subsection{Hyperintersection}
% ---------------------------------------

For any $I, J\subset \cI$ with $Z_I, Z_J\neq \emptyset$, we choose points $b_I\in Z_I^{*}, b_J\in Z_J^{*}$, and a directed path $\gamma\subset \hat{B}$ starting at $b_I$ and ending at $b_J$. In this section, we will define the hyperintersection of CMC's $\psigma_I\cap_{\gamma}\pmb{\sigma_J}$.

\subsubsection{Indexed path and scissors}

\begin{definition}
The $\cI$-index of the path $\gamma$ is the sequence of strata $Z_K^{*}$ $\gamma$ passes through. Denote the index sequence as $\mathrm{Ind}(\gamma)=\langle I=I_0\rightarrow I_1\rightarrow ...\rightarrow I_m=J\rangle$.
\end{definition}

\begin{definition}
    We say $\gamma$ has convex index seqeuence if for any subsequence $I_{l-1}\rightarrow I_l\rightarrow I_{l+1}$ of $\mathrm{Ind}(\gamma)$, we have either $I_{l-1}, I_{l+1}\leq I_l$ or $I_{l-1}, I_{l+1}\geq I_l$, where ``$\leq, \geq$" is the partial order given by the containment relation $I\leq J\Leftrightarrow I\subset J$.
\end{definition}

Note that for any $I_{l-1}\rightarrow I_l\rightarrow I_{l+1}$ with $I_{l-1}\leq I_{l}\leq I_{l+1}$ (or the other direction), we may rewrite it as $I_{l-1}\rightarrow I_{l-1}\rightarrow  I_l\rightarrow I_l\rightarrow I_{l+1}$ which satisfies the convex index sequence condition. Therefore, we may assume every path $\gamma$ is indexed by convex sequence, written as 
\begin{equation}\label{Eqn:convexsequence}
\begin{tikzpicture}[baseline=(current bounding box.center)]
  \node (S) at (-0.8, 0) {$I$};
  \node (S0) at (-0.4,0) {$=$};
  \node (A) at (0,0) {$I_0$};
  \node (B) at (0.8,0.5) {$I_1$};
  \node (C) at (1.6,0) {$I_2$};
  \node (D) at (2.4,0.5) {$...$};
  \node (E) at (3.2,0) {$I_m$};
  \node (T0) at (3.6,0) {$=$};
  \node (T) at (4,0) {$J$};

  \draw[->] (A) -- (B);
  \draw[->] (B) -- (C);
  \draw[->] (C) -- (D);
  \draw[->] (D) -- (E);
\end{tikzpicture}
\end{equation}
where a triple $I_{l-1}\rightarrow I_{l}\rightarrow I_{l+1}$ with $I_l$ lying above (resp. below) $I_{l-1}$ and $I_{l+1}$ means $I_{l-1}, I_{l+1}\leq I_l$ (resp. $I_{l-1}, I_{l+1}\geq I_l$).

\begin{definition}
    A triple $I_{l-1}\rightarrow I_{l}\rightarrow I_{l+1}$ with $I_{l-1}, I_{l+1}\leq I_l$ is called a scissors.
\end{definition}

Therefore, after adding paths of the form $I\rightarrow I$ if necessary, we may also write the index sequence of $\gamma$ as 
\[
\mathrm{Ind}(\gamma)=\langle S_0\rightarrow S_1\rightarrow ...\rightarrow S_k\rangle
\]
where $S_i$'s are all scissors passed by $\gamma$ in order. We call this the scissors sequence of $\gamma$.

\subsubsection{Definition of hyperintersection}

Given a path $\gamma$ connecting $x_I$ and $x_J$ and with scissors sequence $\langle S_0\rightarrow S_1\rightarrow ...\rightarrow S_k\rangle$, we will define the hyperintersection $\psigma_I\cap_{\gamma}\psigma_J$ as a subcone of both $\psigma_I$ and $\psigma_J$. Note that the hyperintersection does not depend on the choice of $b_I\in Z_I^{*}$ and $b_J\in Z_J^{*}$.

\begin{definition}\label{Def:hyperint}
    The hyperintersection $\psigma_I\cap_{\gamma}\psigma_J$ is defined inductively on the number of scissors as follows.
    \begin{enumerate}
    \item When $\gamma$ only contains one scissors $S: I\rightarrow K\rightarrow J$, by the adjacent maps $\psigma_I, \psigma_J\hookrightarrow \psigma_K$, we may define $\psigma_I\cap_{\gamma}\pmb{\sigma}_J=\psigma_I\cap_{S}\pmb{\sigma}_J:= \psigma_I\cap\pmb{\sigma}_J$, viewed as a subcone of all of $\psigma_I, \psigma_J$ and $\psigma_K$.

     \item Suppose the hyperintersection $\psigma_I\cap_{\gamma'}\psigma_{J'}$ is defined for all $I,J'$ and $\gamma'$ with no more than $k-1$ scissors. For $\psigma_I\cap_{\gamma}\psigma_J$ where $\gamma$ contains $k$ scissors, we may define 
    \[
    \psigma_I\cap_{\gamma}\psigma_J:=(\psigma_I\cap_{\gamma'}\psigma_{J'})\cap_{S_k}\psigma_J
    \]
    where $\gamma'$ connects $I$ and $J'$ with $k-1$ scissors and $J'$ and $J$ is connected by a single scissor $S_k$, and $\psigma_I\cap_{\gamma'}\psigma_{J'}$ is viewed as a subcone of $\psigma_I$ and $\psigma_{J'}$.
    \end{enumerate}
    In particular, $\psigma_I\cap_{\gamma}\psigma_J$ may be realized as a subcone of $\psigma_K$ for any $K$ appears in the index sequence of $\gamma$. We call this a representation of $\psigma_I\cap_{\gamma}\psigma_J$ in $\psigma_K$.
\end{definition}

\begin{remark}
    When an index $K$ appears in $\mathrm{Ind}(\gamma)$ for multiple times, there are in general multiple ways to choose representations of $\psigma_I\cap_{\gamma}\psigma_J$ in $\psigma_K$ due to the effect of monodromy. We will discuss this phenomena in the last section. In particular, the natural embedding of $\psigma_I\cap_{\gamma}\psigma_J$ in $\psigma_I$ and $\psigma_J$ refers to the representation corresponding to $I$ appearing at the head and the $J$ appearing at the tail unless specified.  
\end{remark}

\begin{definition}\label{Def:linkedrep}
    If $\psigma\subset \psigma_I, \psigma'\subset\psigma_J$ are the representations of some hyperintersection $\psigma_I\cap_{\gamma}\psigma_J$, we call $(\psigma, \psigma')$ as a pair of \textit{linked} representations.
\end{definition}

% \begin{remark}
%     \jacob{This remark confuses me. {\color{blue} Haohua -- Ok, I just want to remind the definition of being ``simplicial".}{\color{purple} that's fine, but what do you mean "generalization of intersections of cones in a polyhedral complex? I think I would remove the second sentence, I am fine with the first.}{\color{blue}Haohua -- It is an intuitive way to say that hyperintersections are not traditional intersections in general: We do not have to realize two cones in a common linear space, nor do we require them intersect physically. It is unnecessary to say this though...}}For the case when $\hat{B}$ is simplicial corresponds to $\bV\rightarrow B$ (Definition \ref{Def:simplicialvhs}), it is equivalent to say that any hyperintersection $\psigma_I\cap_{\gamma}\psigma_J$ is a face of both $\psigma_I$ and $\psigma_J$ in terms of their representatives. In this sense, hyperintersection and simplicial condition may be seen as generalizations of intersection of cones in a polyhedral complex.
% \end{remark}

The main theorem for this section is:

\begin{theorem}\label{Thm:finitehypint}
    Running over all indices $J\subset \cI$ and all (finitely) indexed paths $\gamma\subset \hat{B}$ connecting $Z_I^{*}$ and $Z_J^{*}$, there are only finitely many different subcones of $\psigma_I$ acting as the representative of some hyperintersection $\psigma_I\cap_{\gamma}\psigma_J$.
\end{theorem}

The theorem will be proved in Section \ref{Sec09}. In the remaining of this section, we will assume Theorem \ref{Thm:finitehypint} and showing how this theorem leads to the existence and construction of a CMCX (hence prove Theorem \ref{Thm:CMCXexists}).

% ---------------------------------------
\subsection{Construction of CMCX}
% ---------------------------------------

The purpose of this section is to prove Theorem \ref{Thm:CMCXexists} by direct construction. 

\subsubsection{Semicomplex and the canonical complexification}

We introduce the concept of polyhedral semicomplex (or simply semicomplex) and its canonical complexification. Note that in this section we assume polyhedral cones are closed.

\begin{definition}\label{Def:semicomplex}
    Let the field $\mathbb{F}$ be $\bQ$ or $\bR$. A collection $\cC$ of cones in $\mathbb{F}^n$ is called a polyhedral semicomplex if the following two conditions hold.
    \begin{enumerate}
        \item For any $\sigma\in \cC$, every face of $\sigma$ is a union of cones in $\cC$.
        \item For any $\sigma, \tau\in \cC$, if $\sigma^{\circ}\cap\tau^{\circ}\neq\emptyset$, then $\sigma=\tau$.
    \end{enumerate}
A \textit{sub-semicomplex} is a subset of cones which also constitutes a semicomplex.

Given a semicomplex $\cC$ and a cone $\tau$ which is a union of faces of $\cC$, we denote by $\cC\!\!\mid_\tau$ the collection of cones in $\cC$ contained in $\tau$, which is also a semicomplex.
\end{definition}

Any polyhedral complex is a semicomplex. Besides, a standard example of semicomplex is the following which will be used later.
\begin{lemma}\label{Lem:StandardSemiComp}
Given a polyhedral complex $\cC'$. For any cone $\sigma\in \cC'$, associate a polyhedral complex $\Sigma(\sigma)$ supported on $\sigma$, then the collection $\cC$ given by replacing each cone $\sigma\in \cC'$ by cones in the minimal common refinement\footnote{If $\cC_1,...,\cC_m$ are polyhedral complexes with the same support, their minimal common refinement complex is the collection $\{\bigcap_k\sigma_k, \ \sigma_k\in \cC_k\}$.} $\hat{\Sigma}(\sigma)$ of $\{\Sigma(\tau)|_{\sigma}, \ \tau\in \cC', \ \sigma\leq \tau\}$ is a semicomplex.
\end{lemma}
\begin{proof}
    Condition (2) of Definition \ref{Def:semicomplex} is clearly satisfied. For any $\tau\in \cC$, there exists a unique $\sigma\in \cC'$ such that $\tau^{\circ}\subset \sigma^{\circ}$, in this case any face $\tau'\leq \tau$ is a cone in $\hat{\Sigma}(\sigma)$ and is contained in a unique minimal face $\sigma'$ of $\sigma$. Since by construction $\hat{\Sigma}(\sigma')$ is a refinement of $\hat{\Sigma}(\sigma)|_{\sigma'}$, $\tau'$ is a union of faces in $\hat{\Sigma}(\sigma')$ hence in $\cC$.
\end{proof}

\begin{figure}[h!]
\centering
\includegraphics[width=0.6\textwidth]{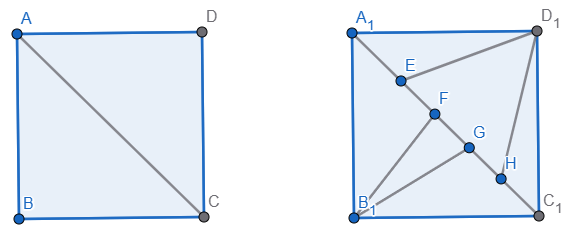}
\caption{Semicomplex from subdividing complex}
\label{SemiCompFromSub}
\end{figure}

For any semicomplex $\cC$, there is a canonical complexification $\overline{\cC}$ of $\cC'$ with the same support and refining every cone in $\cC'$. 

\begin{definition}
    The $r$-skeleton of $\cC$, denoted as $\cC_r$, is the collection of all cones in $\cC$ of dimension $\leq r$.
\end{definition}

Evidently the collection $\cC_1$ is a polyhedral complex, denote $\cC^1:=\cC$. Suppose there is a refinement $\cC^{r-1}$ of $\cC$ such that $\cC^{r-1}_{r-1}$ is a polyhedral complex. Consider any cone $\sigma\in \cC$ of dimension $r$ such that $\cC^{r-1}|_{\sigma}$ is not a polyhedral complex. Suppose $\sigma$ is spanned by $1$-dimensional cones $N_1,...,N_q$, let $G=\frac{1}{q}(N_1+...+N_q)$. There is a unique collection of disjoint cones $\{\tau_1,...,\tau_m\}\subset \cC^{r-1}$ such that $\sqcup_k\tau_k=\partial\sigma$, and $\sigma=\sqcup_k\mathrm{Conv}\{G, \tau_k\}$. We replace every such $\sigma\in \cC^{r-1}$ by the polyhedral complex consists of all $\{\tau_k\}$ and $\{\mathrm{Conv}\{G, \tau_k\}\}$, denote the resulting refinement of $\cC^{r-1}$ by $\cC^r$. It is clear that $\cC^r_r$ is a polyhedral complex. 

\begin{definition}\label{Def:CanonicalComp}
    The canonical complexification $\overline{\cC}$ of $\cC$ is defined to be $\cC^M$, where $M$ is the maximal dimension of cones in $\cC$. 
\end{definition}

\begin{prop}
    $\cC^r_r=\overline{\cC}_r$ for any $0\leq r\leq M$.
\end{prop}
\begin{proof}
    In the refinement process $\cC\rightarrow \cC^2\rightarrow...\rightarrow \cC^M=\overline{\cC}$, any cone of dimension $r$ is not subdivided in the sequence $\cC^r\rightarrow\cC^{r+1}\rightarrow...\rightarrow \cC^M=\overline{\cC}$.
\end{proof}

\begin{lemma}\label{lem:subcomplexification}
    Let $\cC\subset\cE$ be a sub-semicomplex. Then $\ol{\cC}\subset \ol{\cE}$ is a sub-complex.
\end{lemma}

\begin{proof}
    In fact we claim by induction that $\cC^{r}$ is a subcomplex of $\cE^{r}$. For $r=1$ it is clear. Suppose its true for $r-1$. Then let $\sigma\in\cE$ be such that $\cE^{r-1}\mid _\sigma$ is not a polyhedral complex. We then replace $\sigma$ by a union of other cones whose interior is contained in $\sigma^\circ$. If $\sigma\in\cC$ then this process is the same for $\cC^r$ as for $\cE^r$. If however $\sigma\not\in\cC$, then none of the new cones are either. The claim follows.
\end{proof}

\begin{figure}[h!]
\centering
\includegraphics[width=0.6\textwidth]{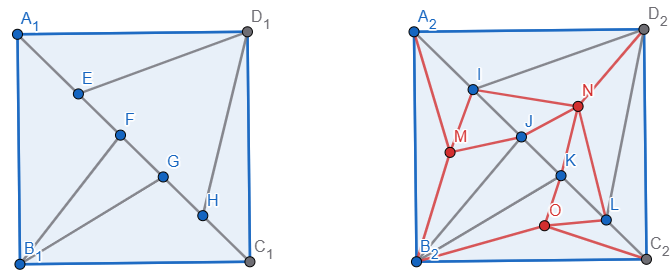}
\caption{The canonical complexification}
\label{CanonicalComp}
\end{figure}

\subsubsection{Proof of Theorem \ref{Thm:CMCXexists}}

We are ready to prove Theorem \ref{Thm:CMCXexists}. Since for every $I\leq J\subset \cI$, the image of the adjacent map $\psigma_I\hookrightarrow\psigma_J$ can be realized as a hyperintersection $\psigma_I\cap_{\gamma}\psigma_J$, Theorem \ref{Thm:CMCXexists} will follow from the following theorem.

\begin{theorem}\label{Thm:StrongerCMCXexists}
    For every $I\subset \cI$ we may associate a polyhedral complex $\Sigma_{\cT}(I)$ supported on $\psigma_I$ such that for any hyperintersection $\psigma_I\cap_{\gamma}\psigma_J$ there exists a polyhedral complex $\Sigma_{\cT}(I, J, \gamma)$ supported on $\psigma_I\cap_{\gamma}\psigma_J$ which is realized as a subcomplex of both $\Sigma_{\cT}(I)$ and $\Sigma_{\cT}(J)$ via the corresponding representatives.
\end{theorem}

\begin{remark}
    Though Theorem \ref{Thm:StrongerCMCXexists} has a stronger statement, it is actually equivalent to Theorem \ref{Thm:CMCXexists}. We leave this to interested readers.
\end{remark}

To prove Theorem \ref{Thm:StrongerCMCXexists}, we introduce the following $3$-step construction of a collection of complexes $\cT$ associating each $I\subset \cI$ a polyhedral complex $\Sigma_{\cT}(I)$ supported on $\psigma_{I}$ satisfying the requirement. 

\medskip

\noindent \textit{Step 1:}  For any $I\subset \cI$ with $Z_I^{*}\neq \emptyset$, associate a polyhedral complex $\Sigma^0_{\cT}(I)$ refining every possible hyperintersection $\psigma_I\cap_\gamma\psigma_J$. Theorem \ref{Thm:finitehypint} guarantees the existence of such finite polyhedral complexes.

\medskip

\noindent \textit{Step 2:} For any $I\subset \cI$ and any $\tau \in \Sigma^0_{\cT}(I)$, consider the polyhedral complex $\Sigma(\tau)$ supported on $\tau$ defined by the following condition: $\Sigma(\tau)$ is the unique minimal refinement of the set of complexes $\{\Sigma^0_{\cT}(K)|_{\tau}, \tau\subset \psigma_I\cap_{\gamma}\psigma_K \ \text{for some} \ \gamma\subset \hat{B}, K\subset\cI\}$. The existence of this finite refinement is again justified by Theorem \ref{Thm:finitehypint}.

\medskip

\noindent \textit{Step 3:} By Lemma \ref{Lem:StandardSemiComp}, Step 1-2 associate every $I\subset \cI$ a semicomplex $\Sigma'_{\cT}(I)$ supported on $\psigma_I$. $\Sigma_{\cT}(I)$ is defined to be the canonical complexification of $\Sigma'_{\cT}(I)$ given by Definition \ref{Def:CanonicalComp}.

\begin{claim}
    The collection $\cT$ defined by associating each $I\subset \cI$ the complex $\Sigma_{\cT}(I)$ satisfies Theorem \ref{Thm:StrongerCMCXexists} and thereby is a CMCX.
\end{claim}

The only thing left to be proved is the following: For any hyperintersection $\psigma_I\cap_{\gamma}\psigma_J$, $\Sigma_{\cT}(I)$ and $\Sigma_{\cT}(J)$ restrict to the same subcomplex on $\psigma_I\cap_{\gamma}\psigma_J$ via the representatives in $\psigma_I$ and $\psigma_J$. Moreover, it is enough to prove $\Sigma'_{\cT}(I)$ and $\Sigma'_{\cT}(J)$ restrict to the same semicomplex on $\psigma_I\cap_{\gamma}\psigma_J$ by Lemma \ref{lem:subcomplexification}.

Note that there is an one to one correspondence between the sets of representatives of hyperintersections in $\psigma_I$ and $\psigma_J$ having non-empty intersection with the representatives of $\psigma_I\cap_{\gamma}\psigma_J$. The equivalence is simply given by composing with $\cap_{\gamma}\psigma_J$ or $\cap_{\gamma^{-1}}\psigma_I$. In other words, the set of hyperintersections intersecting $\psigma_I\cap_{\gamma}\psigma_J$ is canonical and does not depend on the chosen representative. By the construction of Step 3 above, $\Sigma'_{\cT}(I)|_{\psigma_I\cap_{\gamma}\psigma_J}$ and $\Sigma'_{\cT}(J)|_{\psigma_I\cap_{\gamma}\psigma_J}$ are both the minimal common refinement of the same set of complexes obtained by hyperintersections, therefore must be equivalent. This finishes the proof of Theorem \ref{Thm:StrongerCMCXexists}.

% ---------------------------------------
\section{On the finiteness of hyperintersections}\label{Sec09}
%----------------------------------------

The purpose of this section is to fulfill the proof of Theorem \ref{Thm:finitehypint} which will leads to the conclusion of Theorem \ref{Thm:CMCXexists}.

\subsection{Reinterpretation of hyperintersection}\label{Sec10-1}

To show Theorem \ref{Thm:finitehypint}, we need an alternative description on the concept of hyperintersection. 

\begin{lemma}\label{Lem:LocalMonodromy}
    For any $I\in \mathcal{I}$ we may choose a neighborhood $\mathcal{U}_I\subset \hat{B}$ of $Z_I^*$ with the monodromy group $\Gamma_{\cU_I}$ of $\bV|_{\cU_I}\rightarrow \cU_I\cap B$ centralizing the monodromy cone $\sigma_I$.
\end{lemma}
\begin{proof}
    We only have to show the case where $|I|=1$. For general $I$, we just take the intersection of neighborhood we get for each member in $I$. The rest of the proof is similar to Lemma \ref{Lem:localmonodromyCMCX}: For $I=\{i\}$ and any $p\in Z_i^*$ we may choose a punctured coordinate neighborhood $\cU_p\subset B$, a local lift $\Phi_p$ of $\Phi|_{\cU_p}$ such that the monodromy logarithm around $Z_i$ is the fixed $N_i\in\mathfrak{g}_{\bQ}$. Different choices of $\Phi_p$ are differed by the action of the centralizer of $N_i$, therefore the monodromy group of $\cU_i:=\bigcup_{p\in Z_i^*}\cU_p$ must centralize $N_i$.
\end{proof}

Fix a hyperintersection $\psigma_I\cap_{\gamma}\psigma_J$ whose path $\gamma$ has convex index sequence \eqref{Eqn:convexsequence}. We also fix a neighborhood $\mathcal{U}_I$ for every $I\subset \cI$ provided by Lemma \ref{Lem:LocalMonodromy}.

\begin{definition}\label{Def:Companionpath}
    A companion path of $\gamma$, denoted as $\gamma^{\mathrm{cp}}$, is a path in $B$ satisfying the following conditions.
    \begin{enumerate}
        \item For each $0\leq k\leq m$ there exists a point 
        $b_k\in \mathcal{U}_{I_k}\cap S$ such that $\gamma^{\mathrm{cp}}$ passes through $b_0,...,b_m$ in order, and $b_0, b_m$ are endpoints.

        \item If $I_k\leq I_{k+1}$ (resp. $I_k\geq I_{k+1}$), then the part of $\gamma^{\mathrm{cp}}$ connecting $b_k$ and $b_{k+1}$, lies in $\mathcal{U}_{I_k}\cap B$ (resp. $\mathcal{U}_{I_{k+1}}\cap B$).
    \end{enumerate}
    We also denote $\gamma^{\mathrm{cp}}_{\leq k}$ (resp. $\gamma_{\leq k}$) as the part of  $\gamma^{\mathrm{cp}}$ (resp. $\gamma$) connecting $b_0$ and $b_k$ (resp. a smooth point on $Z_{I_0}$ and a smooth point on $Z_{I_k}$).
\end{definition}

For every $0\leq k\leq m$, let $\tilde{\sigma}_k\subset \mathrm{End}(V_{b_k}, \bQ)$ be the monodromy cone obtained from loops based on $b_k$ and contained in $\mathcal{U}_{I_k}\cap B$, and $\sigma_k\subset \mathrm{End}(V_{b_0}, \bQ)$ be the monodromy cone obtained from analytic continuation of $\tilde{\sigma}_k$ from $b_k$ to $b_0$ along $(\gamma^{\mathrm{cp}}_k)^{-1}$.

\begin{theorem}\label{Thm:AlterinterpofHyperint}
   Under the natural identification $\psigma_I\simeq \overline{\sigma_0}$, $\psigma_I\cap_{\gamma}\psigma_J$ isomorphic to $\bigcap_{0\leq k\leq m}\overline{\sigma_k}=\bigcap_{0\leq k\leq m}(\overline{\sigma_0}\cap\overline{\sigma_k})$.
\end{theorem}

\begin{proof}
    Firstly, as a consequence of Lemma \ref{Lem:LocalMonodromy}, $\bigcap_{0\leq k\leq m}\overline{\sigma_k}$ is well-defined and does not depend on the choice of a companion path $\gamma^{\mathrm{cp}}$. The rest of the proof follows directly from the Definitions \ref{Def:hyperint} and \ref{Def:Companionpath} and induction on the number of indexes in $\gamma$.
\end{proof}

\subsection{Proof of Theorem \ref{Thm:finitehypint}}

In this section we prove Theorem \ref{Thm:finitehypint} which will finish the proof of Theorem \ref{Thm:MainTHMformal}. The proof relies on multiple results in the paper \cite{DR23} which we recall below.

\begin{lemma}\cite[Lemma 3.4]{DR23}\label{Lem:DR23lemma}
Suppose that $\sigma,\tau$ are nilpotent cones with $\tilde{A}(\sigma), \tilde{A}(\tau)\neq\emptyset$.  If $\sigma\cap\tau \not=\emptyset$ is open in $\tau$, then $\tilde{A}(\sigma)\subset\tilde{A}(\tau)$.
\end{lemma}

\begin{theorem}\cite[Thm. 2.2 \& Cor. 2.4]{DR23}\label{Thm:DR23finiteness}
    Let $\sigma, \tau$ be two nilpotent cones, then the set
    \[
    \{\sigma\cap\mathrm{Ad}_g\tau , \ g\in \Gamma, \ \tilde{A}(\sigma)\cap \tilde{A}(\mathrm{Ad}_{g}\tau)\neq \emptyset\} 
    \]
    of subcones of $\sigma$ is finite.
\end{theorem}

Fixed any $I\subset \cI$, denote
\begin{align}
    \cS_I&:=\{\psigma\subset\psigma_I, \ \psigma \ \text{is realized by some hyperintersection}\}, \\
    \cS_{I,l}&:=\{\psigma\in \cS, \ \mathrm{codim}_{\psigma_I}(\psigma)\leq l\}.
\end{align}
To show the finiteness of $\cS_I$, we proceed by induction on $l$. As $\cS=\cS_l$ for $l>\mathrm{dim}(\psigma_I)$, this will complete the proof of Theorem \ref{Thm:finitehypint}.

For any hyperintersection $\psigma_I\cap_\gamma \psigma_J$, denote
\[
\mathrm{cd}_I(\gamma,J):=\mathrm{codim}_{\psigma_I}(\psigma_I\cap_\gamma \psigma_J).
\]
We first consider the case $\mathrm{cd}_I(\gamma,J)=0$. By Theorem \ref{Thm:AlterinterpofHyperint}, there exists indexes $I_0=I, I_1,...,I_m=J$ such that 
\begin{equation}\label{Eqn:hyperintaltrep}
    \psigma_I\cap_\gamma \psigma_J\simeq\bigcap_{0\leq k\leq m}(\overline{\sigma_0}\cap\overline{\sigma_k})
\end{equation}
for certain choices of representatives $\sigma_0,...,\sigma_m$ of nilpotent cones around strata $Z_0=Z_I,...,Z_m=Z_J$. $\mathrm{cd}_I(\gamma,J)=0$ implies every term on the RHS of \eqref{Eqn:hyperintaltrep} has the same dimension as $\sigma_I$. Lemma \ref{Lem:DR23lemma} implies $\tA(\sigma_0)=\tA(\sigma_0\cap\sigma_k)$ for any $0\leq k\leq m$, hence all of $\sigma_0\cap\sigma_k$ belongs to the finite set given by Theorem \ref{Thm:DR23finiteness}. In other words, there are only finitely many different possible combinations on the RHS of \eqref{Eqn:hyperintaltrep} which implies there are only finitely many $\psigma_I\cap_\gamma \psigma_J\subset \psigma_I$ with $\mathrm{cd}_I(\gamma,J)=0$. This shows the finiteness of $\cS_{I,0}$.

Now for $l\in\bZ$ we assume the finiteness of $\cS_l$, and we will use this to show the finiteness of $\cS_{l+1}$.

\begin{lemma}\label{Lem:finitedualrep}
    Fix some $\psigma\subset \psigma_I$. Running over all indexes $J\subset \cI$, all possible indexed paths $\gamma$ connecting $I,J$, there are only finitely many $\psigma'$ as a subcone of some $\psigma_J$ such that $(\psigma, \psigma')$ can be realized as a pair of linked representations of some hyperintersection $\psigma_I\cap_\gamma\psigma_J$.
\end{lemma}
\begin{proof}
  Following Section \ref{Sec10-1}, we may choose base points $b_I, b_J\in B$, a companion path $\gamma^{\mathrm{cp}}$ of $\gamma$ such that $\psigma\simeq\psigma'$ is induced by analytically continuing some nilpotent cone $\sigma\subset\sigma_I\subset \mathrm{End}(V_{b_I}, \bQ)$ to some $\sigma_J\subset \mathrm{End}(V_{b_J}, \bQ)$ along $\gamma^{\mathrm{cp}}$. By Lemma \ref{Lem:DR23lemma} and Theorem \ref{Thm:DR23finiteness}, there are only finitely many subcones $\tau\subset \sigma_J$ such that $\tau$ is the image of some analytic continuation of $\sigma$ along some path connecting $b_I$ and $b_J$. This yields the lemma.
\end{proof}

Now we suppose $\psigma_I\cap_\gamma \psigma_J\subset \psigma_I$ is a hyperintersection with $\mathrm{cd}_I(\gamma,J)=l+1$. there exists some $1\leq l\leq m$ such that $I_{l-1}\rightarrow I_l\rightarrow I_{l+1}$ is a scissor, and
\begin{align}
    \mathrm{cd}_I(\gamma_{\leq l-1},J)&\leq l \\
    \mathrm{cd}_I(\gamma_{\leq l},J)&\leq l \\
    \mathrm{cd}_I(\gamma_{\leq l+1},J)&=l+1.
\end{align}
In other words, as subcones of $\psigma_l$, $\psigma_I\cap_{\gamma_{\leq l-1}}\psigma_{I_{l-1}}$ has codimension $\leq l$ while $(\psigma_I\cap_{\gamma_{\leq l-1}}\psigma_{I_{l-1}})\cap \psigma_{l+1}$ has codimension $l+1$. 

Note that the induction hypothesis together with Lemma \ref{Lem:finitedualrep} implies the finiteness of the following set:
\begin{equation}
    \cP_{I,l}:=\{(\psigma, \psigma'), \ \psigma\in \cS_{I,l}, \ \psigma'\in \psigma_K \ \text{for some }K\subset \cI , \ (\psigma, \psigma') \ \text{is a pair of dual representations}\}.
\end{equation}
We may take the $(\psigma, \psigma')\in \cP_{I,l}$ corresponding to $\psigma_I\cap_{\gamma_{\leq l-1}}\psigma_{I_{l-1}}$. The representative $\psigma''\subset \psigma_{l+1}$ of $\psigma_I\cap_{\gamma_{\leq l+1}}\psigma_{I_{l+1}}$ is obtained by cutting $\psigma'$ along a scissor. As a consequence, there are only finitely many possibilities for $\psigma''\subset \psigma_{l+1}$.

Finally, denote $\psigma'''$ as the representative of $\psigma_I\cap_\gamma \psigma_J$ in $\psigma_{l+1}$ and $\overline{\sigma'''}$ be its image in $\mathrm{End}(V_{b_{l+1}}, \bQ)$. It follows that $\psigma'''\subset \psigma''$ and
\begin{equation}\label{Eqn:alterrephypint2}
    \psigma'''\simeq \bigcap_{l+1\leq k\leq m} (\overline{\sigma'''}\cap\overline{\sigma_k}).
\end{equation}
Since every term on the RHS of \eqref{Eqn:alterrephypint2} has dimension $\mathrm{dim}(\psigma_I)-l-1$ which equals to the dimension of $\psigma'''$, The same arguments used to prove the finiteness of $\cS_{l,0}$ show that there are only finitely many possibilities for $\psigma'''$. 

To sum up, for every $\psigma\in \cS_{I,l}$, there are finitely many $\tau'$ acting as its dual representation of dimension $\geq \mathrm{dim}(\psigma_I)-l$, and for every such $\tau'$, there are finitely many $\tau''$ acting as cutting along some scissor with $\mathrm{dim}(\tau'')=\mathrm{dim}(\psigma_I)-l-1$, and for every such $\tau''$, there are finitely many $\tau'''$ acting as its dual representation of the same dimension $\mathrm{dim}(\psigma_I)-l-1$. Therefore, the finiteness of $\cS_{I,l}$ implies the finiteness of $\cS_{I,l+1}$. The induction is complete.

\printbibliography

@article{AB12,
author = {Valery Alexeev and Adrian Brunyate },
title = {Extending the Torelli map to toroidal compactifications of Siegel space},
journal = {Inventiones mathematicae},
volume = {188},
year = {2012},
pages = {175-196}
}

@article{ABE22,
author = {Valery Alexeev and Adrian Brunyate and Philip Engel},
title = {Compactifications of moduli of elliptic K3 surfaces: Stable pair and toroidal},
journal = {Geometry \& Topology},
volume = {26},
year = {2022},
pages = {3525-3588}
}

@article{ACT11,
author = {Daniel Allcock and James A. Carlson and Domingo Toledo},
title = {The Moduli Space of Cubic Threefolds as a Ball Quotient},
journal = {Memoirs of the American Mathematical Society},
volume = {209},
year = {2011}
}

@book{AMRT10,
author = {A. Ash and David Mumford and M. Rapoport and Y. S Tai},
title = {Smooth Compactifications of Locally Symmetric Varieties
},
edition = {2},
publisher = {Cambridge University Press},
year = {2010}
}

@article{BB66,
author = {Baily, W.L. and Borel, A},
title = {Compactification of arithmetic quotients of bounded symmetric domains},
journal = {Annals of Mathematics},
volume = {84},
year = {1966},
number = {3},
pages = {442-528}
}

@article{BBT22,
author = {Benjamin Bakker and Yohan Brunebarbe and Jacob Tsimerman},
title = {o-minimal GAGA and a conjecture of Griffiths},
journal = {Inventiones mathematicae},
volume = {232},
year = {2022},
pages = {163-228}
}

@article{BBKT24,
author = {Benjamin Bakker and Yohan Brunebarbe and Bruno Klingler and Jacob Tsimerman},
title = {Definability of mixed period maps},
journal = {J. Eur. Math. Soc.},
volume = {26},
number = {6},
year = {2024},
pages = {2191-2209}
}

@misc{BFMT25,
      title={Baily--Borel compactifications of period images and the b-semiampleness conjecture}, 
      author={Benjamin Bakker and Stefano Filipazzi and Mirko Mauri and Jacob Tsimerman},
      year={2025},
      eprint={2508.19215},
      archivePrefix={arXiv},
      primaryClass={math.AG},
      url={https://arxiv.org/abs/2508.19215}, 
}

@article{BKT20,
author = { B. Bakker and B. Klingler and J. Tsimerman},
title = {Tame topology of arithmetic quotients and algebraicity of Hodge loci},
journal = {J. Amer. Math. Soc.},
volume = {33},
year = {2020},
pages = {917-939}
}

@misc{CD24,
      title={On the generic degree of two-parameter period mappings}, 
      author={Chongyao Chen and Haohua Deng},
      year={2024},
      eprint={2408.12090},
      archivePrefix={arXiv},
      primaryClass={math.AG},
      url={https://arxiv.org/abs/2408.12090}, 
}

@article{CDK95,
  author       = {Eduardo Cattani and Pierre Deligne and Aroldo Kaplan},
  title        = {On the locus of Hodge classes},
  journal      = {Journal of the American Mathematical Society},
  volume       = {8},
  number       = {2},
  pages        = {483--506},
  year         = {1995}
}

@article{CK82,
author = {Eduardo Cattani and Aroldo Kaplan},
title = {Polarized mixed Hodge structures and the local monodromy of a variation of Hodge structure},
journal = {Inventiones mathematicae},
volume = {67},
year = {1982},
pages = {101–115}
}

@article{CKS86,
author = {Eduardo Cattani and Aroldo Kaplan and Wilfried Schmid},
title = {Degeneration of Hodge Structures},
journal = {Annals of Mathematics},
volume = {123},
number = {03},
year = {1986},
pages = {457-535}
}

@article{Den22,
author = {Haohua Deng},
title = {Extension of period maps by polyhedral fans},
journal = {Advances in Mathematics},
volume = {406},
number = {17},
year = {2022}
}

@misc{DR23,
      title={Completion of two-parameter period maps by nilpotent orbits}, 
      author={Haohua Deng and Colleen Robles},
      year={2023},
      eprint={2312.00542},
      archivePrefix={arXiv},
      primaryClass={math.AG},
      url={https://arxiv.org/abs/2312.00542}, 
}

@misc{GGR21,
      title={Towards a maximal completion of a period map}, 
      author={Mark Green and Phillip Griffiths and Colleen Robles},
      year={2021},
      eprint={2010.06720},
      archivePrefix={arXiv},
      primaryClass={math.AG}
}

@article{HT14,
author = {Shinobu Hosono and Hiromichi Takagi},
title = {Determinantal Quintics and Mirror Symmetry of Reye Congruences},
journal = {Communications in Mathematical Physics},
volume = {329},
year = {2014},
pages = {1171-1218}
}

@article{HT18,
author = {Shinobu Hosono and Hiromichi Takagi},
title = {Movable vs Monodromy Nilpotent Cones of Calabi–Yau Manifolds},
journal = {Symmetry, Integrability and Geometry: Methods and Applications},
volume = {14},
year = {2018},
number = {039},
pages = {37}
}

@misc{GGLR20,
      title={Period mappings and properties of the augmented Hodge line bundle}, 
      author={Mark Green and Phillip Griffiths and Radu Laza and Colleen Robles},
      year={2023},
      eprint={1708.09523},
      archivePrefix={arXiv},
      primaryClass={math.AG},
      url={https://arxiv.org/abs/1708.09523}, 
}

@article{Gri70,
   author = {Phillip Griffiths},
   title={Periods of integrals on algebraic manifolds: Summary of main results and discussion of open problems},
   journal={Bull. Amer. Math. Soc.},
   volume = {76},
   number = {2},
year = {1970},
pages = {228-296}
}

@article{GRT14,
   author = {Phillip Griffiths and Colleen Robles and Domingo Toledo},
   title={Quotients of non-classical flag domains are not
algebraic},
   journal={Algebraic Geometry},
   volume = {1},
year = {2014},
pages = {1-13}
}

@incollection{Hay14,
  author       = {Tatsuki Hayama},
  title        = {Kato-Usui partial compactifications over the toroidal compactifications of Siegel spaces},
  booktitle    = {Hodge Theory, Complex Geometry, and Representation Theory},
  series       = {Contemporary Mathematics},
  volume       = {608},
  pages        = {143--155},
  publisher    = {American Mathematical Society},
  year         = {2014},
}

@incollection{Kl22,
  author       = {Bruno Klingler},
  title        = {Hodge Theory, Between Algebraicity and Transcendence},
  booktitle    = {Proceedings of the International Congress of Mathematicians (ICM) 2022},
  volume       = {3},
  pages        = {2250--2284},
  publisher    = {European Mathematical Society},
  year         = {2022}
}

@book{KKMSD89,
author = {G. Kempf and F. Knudsen and D. Mumford and B. Saint-Donat},
title = {Toroidal Embeddings 1},
series = {Lecture Notes in Mathematics },
volume = {339},
publisher = {Springer Berlin, Heidelberg},
year = {1973}
}

@article{Abr2002,
  title={Torification and factorization of birational maps},
  author={Abramovich, Dan and Karu, Kalle and Matsuki, Kenji and W{\l}odarczyk, Jaros{\l}aw},
  journal={Journal of the American Mathematical Society},
  volume={15},
  number={3},
  pages={531--572},
  year={2002}
}

@misc{KNU10,
      title={Classifying spaces of degenerating mixed Hodge structures, III: Spaces of nilpotent orbits}, 
      author={Kazuya Kato and Chikara Nakayama and Sampei Usui},
      year={2010},
      eprint={1011.4353},
      archivePrefix={arXiv},
      primaryClass={math.AG}
}

@article{kl17,
  title={Hodge loci and atypical intersections: Conjectures},
  author={Klingler, Bruno},
  journal={Motives and Complex Multiplication},
  year={2016}
}

@article{KO21,
  author       = {Bruno Klingler and Ania Otwinowska},
  title        = {On the Closure of the Hodge Locus of Positive Period Dimension},
  journal      = {Inventiones mathematicae},
  volume       = {225},
  number       = {3},
  pages        = {857--883},
  year         = {2021}
}

@article{KP16,
author = {M. Kerr and G. Pearlstein},
title = {Boundary components of Mumford–Tate domains},
journal = {Duke Math. J.},
volume = {165},
number = {4},
year = {2016},
pages = {661-721}
}

@article{KPR19,
author = {M. Kerr and G. Pearlstein and C. Robles},
title = {Polarized relations on horizontal SL(2)'s},
journal = {Documenta Mathematica},
volume = {24},
year = {2019},
pages = {1295-1360}
}

@book{KU08,
author = {Kazuya Kato and Sampei Usui},
title = {Classifying Spaces of Degenerating Polarized Hodge Structures},
series = {Annals of Mathematics Studies},
volume = {169},
publisher = {Princeton University Press},
year = {2008}
}

@article{Laz10,
author = {Radu Laza},
title = {The moduli space of cubic fourfolds via the period map},
journal = {Annals of Mathematics},
volume = {172},
year = {2010},
pages = {673-711}
}

@article{Nam76,
author = {Yukihiko Namikawa},
title = {A New Compactification of the Siegel Space and Degeneration of Abelian Varieties. I, II},
journal = {Mathematische Annalen},
volume = {221},
year = {1976},
}

@article{Sch73,
author = {Wilfried Schmid},
title = {Variation of Hodge Structure: The Singularities of the Period Mapping},
journal = {Inventiones mathematicae},
volume = {22},
year = {1973},
pages = {211-320}
}

@article{Usu06,
author = {Sampei Usui},
title = {Images of extended period maps},
journal = {J. Algebraic Geom.},
volume = {15},
year = {2006},
pages = {603-621}
}
\end{document}